\documentclass[10pt]{preprint}

\usepackage[margin=2.5cm]{geometry}
\usepackage[utf8x]{inputenc}
\usepackage[english]{babel}
\usepackage[full]{textcomp}

\usepackage[osf]{newtxtext}

\usepackage{amssymb}
\usepackage{amsmath}
\usepackage{amsthm}

\usepackage{mhequ}
\usepackage{booktabs}
\usepackage{tikz}
\usepackage{mathrsfs}
\usepackage[noadjust]{cite}
\usepackage{microtype}
\usepackage{comment}
\usepackage{slashed}
\usepackage{mathtools}
\usepackage{centernot}
\usepackage{footnote}
\usepackage{enumerate}
\usepackage[shortlabels]{enumitem}
\usepackage{stackrel}
\usepackage{longtable}
\usepackage{cprotect}
\usepackage{xstring}
\usepackage{calc}
\usepackage{bbm}
\usepackage{dsfont}

\usepackage[colorlinks=true, pdfstartview=FitV, linkcolor=colorLink, citecolor=colorCite, urlcolor=colorLink, linktocpage=true]{hyperref}

\usepackage[capitalize]{cleveref}

\usepackage{upgreek}
\usepackage{parskip}
\usepackage{graphicx}
\usepackage{orcidlink}

\makeatletter
\renewcommand{\paragraph}{%
\@startsection{paragraph}{4}%
{\z@}{1.5ex \@plus 1.5ex \@minus .2ex}{-0.7em}%
{\normalfont\normalsize\bfseries}%
}
\makeatother

\makeatletter
\def\thm@space@setup{%
  \thm@preskip=\parskip \thm@postskip=0pt
}
\makeatother

\setlist[itemize]{leftmargin=5mm}

\DeclareSymbolFont{timesoperators}{T1}{ptm}{m}{n}
\SetSymbolFont{timesoperators}{bold}{T1}{ptm}{b}{n}
\DeclareMathAlphabet{\mathbb}{U}{jkpsyb}{m}{n}
\SetMathAlphabet{\mathbb}{bold}{U}{jkpsyb}{bx}{n}

\allowdisplaybreaks
\definecolor{colorLink}{RGB}{0,100,162}
\definecolor{colorCite}{RGB}{8,124,100}

\let\alpha\upalpha
\let\beta\upbeta
\let\delta\updelta
\let\gamma\upgamma
\let\mu\upmu
\let\eta\upeta
\let\nu\upnu
\let\rho\uprho
\let\chi\upchi
\let\xi\upxi
\let\zeta\upzeta
\let\tau\uptau
\let\varphi\upvarphi
\let\lambda\uplambda
\let\theta\uptheta
\let\pi\uppi
\let\Upsilon\Upupsilon
\let\Theta\Uptheta
\let\Psi\Uppsi
\let\Xi\Upxi

\def\dash{\leavevmode\unskip\kern0.18em--\penalty\exhyphenpenalty\kern0.18em}
\def\slash{\leavevmode\unskip\kern0.15em/\penalty\exhyphenpenalty\kern0.15em}

\makeatletter
\renewcommand{\operator@font}{\mathgroup\symtimesoperators}
\makeatother

\makeatletter
\DeclareRobustCommand{\TitleEquation}[2]{\texorpdfstring{\StrLeft{\f@series}{1}[\@firstchar]$\if%
b\@firstchar\boldsymbol{#1}\else#1\fi$}{#2}}
\makeatother

\makeatletter
\newcommand{\pushright}[1]{\ifmeasuring@#1\else\omit\hfill$\displaystyle#1$\fi\ignorespaces}
\newcommand{\pushleft}[1]{\ifmeasuring@#1\else\omit$\displaystyle#1$\hfill\fi\ignorespaces}
\makeatother

\renewcommand{\bar}{\overline}
\renewcommand{\hat}{\widehat}
\renewcommand{\tilde}{\widetilde}

\newcommand{\ceil}[1]{\lceil#1\rceil}

\makeatletter
\newcommand{\oset}[3][0ex]{%
  \mathrel{\mathop{#3}\limits^{
    \vbox to#1{\kern-2\ex@
    \hbox{$\scriptstyle#2$}\vss}}}}
\makeatother

\usetikzlibrary{shapes.misc}
\usetikzlibrary{shapes.symbols}
\usetikzlibrary{shapes.geometric}
\usetikzlibrary{decorations}
\usetikzlibrary{decorations.markings}
\usetikzlibrary{calc}
\usetikzlibrary{external}
\usetikzlibrary{arrows}
\usetikzlibrary{patterns}

\usepackage[colorinlistoftodos,prependcaption,textsize=footnotesize, color=red!20]{todonotes}

\usepackage{leftindex}
\usepackage{float}

\newcommand{\dd}{\mathrm{d}}
\newcommand{\g}{\tilde{g}}
\DeclareMathOperator{\Div}{div}

\usepackage{amsmath,amsfonts}

\newcommand{\e}{\mathrm{e}}
\newcommand{\E}{\mathbb{E}}
\newcommand{\F}{\mathcal{F}}
\newcommand{\N}{\mathbb{N}}

\newcommand{\R}{\mathbb{R}}
\newcommand{\T}{\mathbb{T}}
\newcommand{\U}{\mathcal{U}}
\newcommand{\PP}{\mathbb{P}}
\newcommand{\one}{\mathbbm{1}}
\newcommand{\gR}{\gamma_R( \|u\|_{H^k})}
\newcommand{\gRp}{\gamma_R( \|\bar{u}\|_{H^k})}

\newcommand{\grad}{\nabla}

\newcommand{\intT}{\int_{\T^d}}
\newcommand{\csubset}{\subset\joinrel\subset}
\newcommand{\inttn}{\int_0^{t\wedge \sigma_n}}

\newcommand{\LUH}{{\mathcal{L}_2(U,H)}}

\newcommand{\falt}{\ast\eta_\varepsilon}
\newcommand{\faltl}{\ast\eta_{\varepsilon_{m'_l}}}
\newcommand{\sumi}{\sum_{\iota\geq 1}}

\newcommand{\gv}{g_\iota(v\ast\eta_\varepsilon)}

\newcommand{\gdv}{g_\iota'(v\ast\eta_\varepsilon)}

\newcommand{\uml}{u_{\varepsilon_{m'_l}}}

\newcommand{\ue}{u^\varepsilon}
\newcommand{\uem}{u^{\varepsilon,m}}
\newcommand{\uer}{u^{\varepsilon,R}}
\newcommand{\uerp}{\bar{u}^{\varepsilon,R}}
\newcommand{\ur}{u^{R}}
\newcommand{\up}{\bar{u}}
\newcommand{\VV}[1]{\leftindex_{V^*}{\langle} #1 \rangle_V}
\newcommand{\HHH}[1]{\langle #1 \rangle_{H}}
\newcommand{\trn}{{\tau_R^{(n)}}}
\newcommand{\M}{\mathcal{M}(\R)}

\newtheorem{theorem}{Theorem}[section]
\newtheorem*{theorem*}{Theorem}
\newtheorem{corollary}[theorem]{Corollary}
\newtheorem{lemma}[theorem]{Lemma}
\newtheorem*{lemma*}{Lemma}

\newtheorem{defi}[theorem]{Definition}

\theoremstyle{definition}

\newtheorem{remark}[theorem]{Remark}
\newtheorem{assumption}[theorem]{Assumption}

\newtheorem{cntexample}[theorem]{Counterexample}
\newtheorem*{cntexample*}{Counterexample}

\crefname{defi}{Definition}{Definitions}
\crefname{setting}{Setting}{Settings}
\crefname{cntexample}{Counterexample}{Counterexamples}
\crefname{hypo}{Hypothesis}{Hypotheses}
\crefname{assumption}{Assumption}{Assumptions}

\title{Non-local to local limit for stochastic conservation laws}

\author{Sarah Geiss$^{1}$, Benjamin Gess$^{2}$}

\institute{
    Institut f\"{u}r Mathematik, Technische Universit\"{a}t Berlin, Email:
    \href{mailto:geiss@math.tu-berlin.de}{geiss@math.tu-berlin.de}
    \and
    Institut f\"{u}r Mathematik, Technische Universit\"{a}t Berlin \& Max–Planck–Institute for Mathematics in the Sciences, Leipzig, Email: \href{mailto:benjamin.gess@tu-berlin.de}{benjamin.gess@tu-berlin.de} 
    }

\date{October 26, 2026}   

\begin{document}

\maketitle

\bigskip
\begin{abstract}
This work proves the convergence of stochastic nonlocal conservation laws, with spatially inhomogeneous noise, to stochastic conservation laws when the range of particle interaction vanishes. This provides rigorous justification for a class of  stochastic conservation laws that has received considerable attention in the last decade. To handle technical challenges  such as gradient-dependent noise, modified versions of the equations with specific cutoff boundaries are utilized and new uniform estimates to ensure stability during the limiting process are established. In addition, the work demonstrates that this convergence depends on the presence of a viscous term, as a counterexample shows that without it, the limiting solution may fail to be an entropy solution.
\end{abstract}

\tableofcontents

\section{Introduction}
\label{sec:1}
Over the last decade, stochastic scalar conservation laws
    \begin{equation}\label{eq:intro_SSCL}
    du=\nu \Delta u~dt- \nabla \cdot(u A(u,x) \circ d\xi ),
    \end{equation}     
where $\xi$ is white-in-time, spatially correlated noise, have received significant interest, in particular in their non-viscous form $\nu =0$, including the development of well-posedness theories for kinetic and entropy solutions, the analysis of regularity and long-time behavior, cf. for example \cite{LionsPerthameSouganidis, Lions2014, Gess2015, Friz2016, Deya2019, Karlsen2017, Karlsen2018, Chouk2017, Gess2024, Fehrman2023, Gess2017, Ito-DebusscheHofmanovaVovelle}. Informally, these local-in-space stochastic equations are derived as the singular local limit of nonlocal interactions where the flux is determined by a convolution. Specifically, the local term $\nabla \cdot(u A(u,x))$ is understood as the formal limit of the nonlocal term $\nabla \cdot(u A(u * \eta_\varepsilon,x))$ as the kernel $\eta_\varepsilon$ concentrates toward a Dirac mass. While significant progress has been made in the understanding of stochastic scalar conservation laws like \eqref{eq:intro_SSCL}, the rigorous passage from nonlocal-in-space SPDEs to their local-in-space counterparts has not yet been established. The main focus of this work is to provide a first rigorous mathematical justification of this.

Interacting particle systems subject to common noise appear in various settings, including the training of overparameterized neural networks \cite{GessKonarovskyi23, RotskoffVanden2022}, interacting solvents \cite{Kotelenez95, CoghiGess}, nonlinear filtering \cite{KurtzXiong99}, opinion dynamics \cite{zimper2025mean}, and mean field games with common noise \cite{Cardaliaguet19}. In these contexts, the macroscopic evolution is often described by nonlocal SPDEs
    \begin{equation}\label{eq:intro_NSSCL}
    du^{\varepsilon}=\nu \Delta u^{\varepsilon}~dt- \nabla \cdot(u^{\varepsilon}A(u^{\varepsilon}*\eta_{\varepsilon},x) \circ d\xi ),
    \end{equation} 
where the interaction between particles is mediated by a convolution kernel $\eta_\varepsilon$ of length $\varepsilon$. The central hypothesis in the derivation of  stochastic conservation laws is that if the interaction length converges to zero, these nonlocal SPDEs converge to stochastic conservation laws.

In the following, let $\mathbb{T}^d$ denote the  $d$-dimensional flat torus. Informally, the main result obtained in this work is as follows; the detailed statement with precise assumptions is given in Theorem \ref{thm:main} below.

\begin{theorem*}[Theorem \ref{thm:main} below] Assume sufficient regularity of the diffusion coefficient $A$ and the initial condition $u_{0}\ge0$, and impose structural assumptions on the convolution kernel $\eta_{\varepsilon}$. Set $k= \lceil d/2 + 3/2\rceil$. Let $(u, \sigma)$ denote the unique maximal solution of the SPDE \eqref{eq:intro_SSCL}, with corresponding localizing sequence $(\tilde{\sigma}_{m})_{m\in\mathbb{N}}$. Moreover, let $u^{\varepsilon}$ denote the unique global maximal solution of the nonlocal SPDE \eqref{eq:intro_NSSCL}. Then, for every $m\in\mathbb{N}$, in probability,
\begin{equation}
\lim_{\varepsilon\rightarrow0} \sup_{t\in[0,\tilde{\sigma}_{m}\wedge T]} \|u^{\varepsilon}(t)-u(t)\|_{H^{k}(\mathbb{T}^{d})}=0. \end{equation}
\end{theorem*}

In the deterministic context, the local limits of nonlocal conservation laws have gained significant interest due to applications in traffic and crowd dynamics, particularly in the non-viscous setting (see Subsection~\ref{subsec11} for references and further discussion). The stochastic case, however, differs substantially from the deterministic one. This is due to the change of the nature of the convective term in \eqref{eq:intro_SSCL}, \eqref{eq:intro_NSSCL}: While, in the deterministic setting, the convection is a lower order term, and, thus, positive regularity is readily obtained in the viscous case $\nu>0$, in the stochastic setting, the stochastic convection scales at the same order as the viscous term. As a consequence, even on an informal level, it is not clear which term dominates the dynamics. In particular, it is unknown if solutions to \eqref{eq:intro_SSCL} can develop shocks despite the viscosity in \eqref{eq:intro_SSCL}, even in the viscous and spatially homogeneous case. In addition, on a more technical level, the Ito-Stratonovich correction in \eqref{eq:intro_NSSCL} leads to a nonlocal second order correction term, which makes it significantly challenging to derive uniform a priori estimates. This is the central challenge addressed in this work, and novel local-in-time estimates that are uniform in $\varepsilon$ are obtained.

An additional technical difficulty stems from the fact that the estimates uniform in $\varepsilon$ are available only locally in time. A direct application of standard arguments would therefore lead to $\varepsilon$-dependent stopping times, whose behavior would have to be controlled when passing to the limit $\varepsilon\to0$. We avoid this obstacle by introducing a family of tamed SPDEs  incorporating a cutoff function $\gamma_R(\|u\|_{H^k})$.

The presence of the viscous term $\Delta u dt$ in \eqref{eq:intro_SSCL}, \eqref{eq:intro_NSSCL} is a critical requirement for this convergence, as it has been noticed in the deterministic case in  \cite[Counterexample 5.5]{ColomboCrippaSpinolo}. In the present work, we generalize their counterexample to the stochastic setting:
\begin{cntexample*}(\cref{counterexample} below) Let $T>0$ and assume structural assumptions on the convolution kernel $\eta_\varepsilon$. Let $u^\varepsilon$ denote the nonlocal solution to
  $$du^{\varepsilon}=-\partial_{x}(u^{\varepsilon}(u^{\varepsilon}*\eta_{\varepsilon}))\circ dW_{t},\quad u^\varepsilon_{0}=\mathbb{I}_{[0,1]}$$
and $u$ the entropy solution to
  $$du =-\partial_{x}(u^2)\circ dW_{t},\quad u_{0}=\mathbb{I}_{[0,1]}.$$
Then for all $p>1$ and $(\varepsilon_{l})_{l\in\N} \subseteq (0,1)$ satisfying $\lim_{l\to\infty}\varepsilon_l = 0$  we have
\begin{equation*}
 \PP\left[\lim_{l\to\infty}\|u^{\varepsilon_l}(T) - u(T)\|_{L^p(\R)} =0 \right] =0.
 \end{equation*}
\end{cntexample*}
This failure is evidenced by the fact that $u^\varepsilon$ preserves the $L \log L$ entropy while the entropy solution $u$ strictly decreases it. 

\textbf{Overview of the paper: }
Section~\ref{sec:1} introduces the background of stochastic scalar conservation laws and situates the work within the existing literature. Section~\ref{sec:2} establishes the precise mathematical setting and assumptions and states the main theorem regarding the convergence of nonlocal solutions. Section~\ref{sec:3} introduces tamed SPDEs utilizing cutoff functions to overcome technical obstacles related to noncoercive coefficients. Section~\ref{sec:4} provides regularity theory for the original, non-tamed equations.
The main technical contribution of the paper is presented in  Section~\ref{sec:5}, where we
establish key stability estimates for the tamed nonlocal SPDEs that remain uniform
as the kernel interaction length $\varepsilon$ tends to zero. These estimates play the central role
in the convergence analysis. Building on these results, Section~\ref{sec:6} completes the
proof of the main theorem using compactness arguments and the Skorokhod
representation theorem. Section~\ref{sec:7} provides a counterexample demonstrating that convergence fails in the inviscid setting.

\subsection{Literature}\label{subsec11}
\textbf{Local limits of mean field systems:} The question of convergence of nonlocal conservation laws to their local counterpart was raised by \cite{Amorim2015} in the context of numerical approximations. The rigorous analysis of the singular local limit was initiated by \cite{ColomboCrippaSpinolo}, who exhibited counterexamples to convergence in the non-viscous setting while proving a positive result in the viscous case.  Hence, subsequent results have primarily concerned the non-viscous problem, and, being motivated by traffic models, they focus on anisotropic kernels. Positive convergence results were obtained by \cite{KeimerPflug2019} for monotone initial data and by \cite{Colombo2021} for initial conditions satisfying a one-sided Lipschitz condition and being bounded away from $0$ using uniform total variation estimates;
the latter work also showed that the total variation may blow up when the assumptions on the initial datum are weakened. This analytical framework was expanded by \cite{Bressan2020, Bressan2021}, who established convergence to the entropy solution for exponential kernels by obtaining uniform total variation bounds through a reformulation as a $2\times2$ hyperbolic relaxation system. Further technical breadth was added by \cite{Coclite2021, Coclite2023}, who resolved the ``doubly singular'' vanishing viscosity limit and extended convergence results for exponential kernels under general initial data.Beyond exponential kernels, \cite{ColomboCrippaMarconiSpinolo2023} established convergence for convex anisotropic kernels under weaker assumptions on the initial data, together with an entropy admissibility criterion and a convergence rate. More recently, \cite{Friedrich2024} established the singular limit for a structurally different model with nonlocal velocity, \(v(u)*\eta_\varepsilon\), using uniform BV estimates on the nonlocal velocity, while \cite{Friedrich2025} extended the nonlocal-to-local analysis to optimal control and proved \(\Gamma\)-convergence of the corresponding control problems. Simultaneously, \cite{marconi2025nonlocal} extended the theory to \(2\times2\) systems, including the Aw–Rascle–Zhang model, a second-order macroscopic traffic-flow model. Nonlocal Oleinik-type estimates have been obtained by \cite{Coclite2025}, while \cite{Chiarello2026} generalized the interaction mechanism to nonlinear $L^p$ norms, and rough fluxes have been treated in \cite{Aggarwal2025}. More recently, \cite{CocliteDeNittiHuang2025} proved convergence to the entropy solution for certain non-convex kernels using entropy-production estimates and compensated compactness.

\textbf{Stochastic scalar conservation laws: } The field of stochastic scalar conservation laws with rough stochastic fluxes was inaugurated by Lions, Perthame, and Souganidis, who established kinetic formulations and $L^1$-contraction for rough fluxes \cite{LionsPerthameSouganidis}, a framework later extended to spatially dependent media \cite{Lions2014}. This foundation was broadened to include multiple independent signals by Gess and Souganidis \cite{Gess2015} and transport noise by Friz and Gess \cite{Friz2016}. Significant advancements in regularity and computation were led by Hoel, Karlsen, Risebro and Storr{\o}sten, who analyzed spatial Lipschitz continuity \cite{Karlsen2017} and provided rigorous foundations for numerical methods \cite{Karlsen2018}. In parallel, the well-posedness/regularization by noise perspective demonstrates that stochasticity can restore uniqueness \cite{Gess2018} and improve regularity \cite{Chouk2017,Gess2017}, and stabilize long-time behavior \cite{Gassiat2024}. To handle stochastic fluxes more intrinsically, Deya, Gubinelli, Hofmanova and Tindel developed a theory of rough kinetic solutions independent of flow transformations \cite{Deya2019}. Recent progress includes non-equilibrium large deviations for singular conservative SPDEs \cite{Fehrman2023}, and establishing long-time stability and unique invariant measures \cite{Gess2017}.

\textbf{Mean field systems with common noise: } Mean field systems with common noise appear in situations where shared environmental noise prevents the law of large numbers from collapsing the empirical measure into a deterministic trajectory. An early rigorous treatment of this phenomenon was developed by \cite{Kotelenez95}, who established the existence of solutions for a class of quasilinear SPDEs, and \cite{KurtzXiong1999}, who proved that empirical measures of systems driven by common Brownian motion converge to the stochastic Fokker-Planck equation. The analysis of the limiting stochastic Fokker–Planck equation was further refined by \cite{CoghiGess}, who established well-posedness for merely measure-valued initial conditions, by developing a duality method in the stochastic setting. Subsequently, this line of work has been extended to less regular coefficients in \cite{GessKonarovskyi23} and \cite{bugini2025nonlinear}. The study of singular interactions has been expanded by Nikolaev \cite{nikolaev2025fluctuation,nikolaev2025quantitative}, who derived quantitative relative entropy estimates to provide explicit convergence rates for such systems. 

In the related context of mean-field games, an analytic breakthrough regarding the convergence of large-scale systems was provided by \cite{Cardaliaguet19}, who addressed the convergence problem in the presence of common noise via the master equation. 
In a distinct, pathwise direction, a rough-path formulation of linear–quadratic mean-field games with rough common signals was developed in  \cite{friz2026mean}, establishing well-posedness and stability results.

\section{Viscous setting: Main result on convergence}
\label{sec:2}
Let $(\Omega, \F, \PP, (\F_t)_{t\in[0,\infty)})$ be a filtered probability space satisfying the usual conditions, i.e. $(\F_t)_{t\geq 0}$ is complete and right-continuous. We assume all processes and stopping times to be adapted to this filtered probability space unless stated otherwise.

Let $\T^d$ denote the $d$-dimensional flat torus. We use the standard notation $W^{k,p}(\R^d)$, $W^{k,p}(\T^d)$,  $H^k(\T^d):=W^{k,2}(\T^d)$, $L^p(\T^d):=W^{0,p}(\T^d)$ for the Sobolev spaces and $\|\cdot\|_{\dot{W}^{k,p}(\T^d)}$ for the $W^{k,p}$ seminorm. We occasionally drop the $\T^d$ or $\R^d$ in the notation.

\begin{assumption}\label{assumption}
\begin{enumerate}
\item (Dimension $d$ and order $k$) Let $d\in\N$ denote the dimension of the torus $\T^d$. Let $k= \lceil d/2 + 3/2\rceil = d/2+\alpha\in\N$ where $\alpha = 1.5$ if $d$ is odd and $\alpha = 2$ if $d$ is even.
\item  (Convolution kernel $\eta$) Assume
\begin{equation}\label{eq:EtaProperties}
\eta\in C^\infty(\R^d, [0,\infty)) \,\,  \text{symmetric with compact support,} \quad \int_{\R^d} \eta (x) \dd x = 1.
\end{equation}
For $\varepsilon\in(0,1)$ we set
\begin{equation*}
 \eta_\varepsilon(x):= \frac{1}{\varepsilon^{d}} \eta(x/\varepsilon) \qquad \forall x\in\R^d.
\end{equation*}
For $v\in L^1(\T^d)$ extend $v$ periodically to $\R^d \mapsto \R$ and  define
\begin{equation*}
v\ast\eta(x):= \int_{\R^d} \eta(y) v(x-y) \dd y \qquad \text{ for all } x\in\R^d.
\end{equation*}
  \item (Noise of SPDE) Let $\{(W^\iota_t)_{t\geq 0}, \iota \in \N_{\geq 1}\}$ denote a family of independent 1-dimensional standard Brownian motions.
  \item (Coefficients of SPDE: $g_\iota$, $\g_\iota$) Let $g_\iota \in C^{k+3}(\R\times\T^d, \R^d)$ for all $\iota\in \N_{>0}$. Assume that
  \begin{equation*}
 \sumi \|g_\iota\|^2_{C^{k+3}([-K,K]\times\T^d, \R^d)}<\infty
 \qquad \forall K>0
  \end{equation*}
  and
  \begin{equation*}
\Div(g_\iota(u,x))= 0 \qquad \forall u\in \R, \forall x\in \T^d, \forall \iota\geq 1 
  \end{equation*}
  where the divergence is w.r.t. the $x$ variable. Define $\g_\iota(u, x):= ug_\iota(u,x)$, $g_\iota'(u,x) := \partial_u g_\iota(u,x)$ and $\g_\iota'(u,x) := \partial_u \g_\iota(u,x)$ for all $u\in\R, x\in\T^d$. We frequently drop the $x$ in the notation e.g. $g_\iota(u):= g_\iota(u,x)$.
  
\item (Cut-off function for tamed SPDE) For every $R>0$ let $\gamma_{R}\in C^\infty([0,\infty),\R)$ be such that $\gamma_R|_{[0,R]} \equiv 1$ and $\gamma_R|_{[R+1,\infty)} \equiv 0$ and $\gamma$ is monotone non-increasing. In addition, assume that $u\mapsto \sqrt{\gamma_{R}(u)}$ is Lipschitz continuous. For $R=\infty$ define $\gamma_{\infty}\in C^\infty([0,\infty),\R)$ by $\gamma_\infty \equiv 1$.
\end{enumerate}
\end{assumption}
For $\varepsilon\in(0,1)$, consider the following formal nonlocal-in-space
SPDE:
\begin{equation}\label{eq:SPDEnonlocal}\tag{$NL_\varepsilon$}
\begin{aligned}
\dd \ue & = -  \sum_{\iota \geq 1} \Div(\ue g_{\iota}(\ue\falt,x)) \circ \dd W^\iota_t  +  \Delta \ue \dd t, \qquad  u^\varepsilon(0) = u_0.
\end{aligned}
\end{equation}
We study the limit as $\varepsilon\to0$, in which the interaction length
vanishes and the equation becomes local-in-space:
\begin{equation}\label{eq:SPDElocal}\tag{$L_0$}
\begin{aligned}
\dd u = -  \sum_{\iota \geq 1} \Div(ug_{\iota}(u,x)) \circ \dd W^\iota_t  + \Delta u \dd t, \qquad  u(0) = u_0. 
\end{aligned}
\end{equation}
The Stratonovich notation in \eqref{eq:SPDEnonlocal} and \eqref{eq:SPDElocal}
is formal: throughout, we interpret both equations in the Itô variational
sense. In particular, the Itô drift includes the
Stratonovich-to-Itô correction, which is incorporated in the coefficients
defined below, see \cref{def:solHk}.

The well-posedness of \eqref{eq:SPDElocal} and \eqref{eq:SPDEnonlocal} has been studied in several regularity classes. For the local equation \eqref{eq:SPDElocal}, global-in-time well-posedness at low spatial regularity has been established, for example, in \cite[Theorem 2.7]{DareiotisGess} and \cite[Theorem 1.6]{FehrmanGess2019WP}. Local-in-time well-posedness at higher spatial regularity is covered by the general theory in \cite[Theorem 6.4]{AgrestiVeraar2022-Part1}. For global-in-time well-posedness of the nonlocal equation \eqref{eq:SPDEnonlocal} at low spatial regularity, see, for example, \cite[Theorem 3.5]{KurtzXiong1999} and \cite[Theorems 5.3 and 5.4]{CoghiGess}.

Our analysis requires solutions with spatial regularity at least $H^k(\T^d)$, $k= \lceil d/2 + 3/2\rceil$. We construct maximal solutions in this class for both equations using the variational framework of Liu and Röckner. Alternatively, local well-posedness in this setting could also be obtained by the critical variational framework of \cite[Theorem 3.3]{AgrestiVeraar2024}, with $\varepsilon>0$ fixed in the nonlocal case.  The notion of maximal solution introduced below is essentially consistent with \cite[Definition 3.2]{AgrestiVeraar2024}.

Let $(W_t)_{t\geq 0}$ denote a cylindrical Wiener process in a separable Hilbert space $U$ on the underlying filtered probability space.  Let $(\e_\iota)_{\iota \geq 1}$ denote an orthonormal basis of $U$ and let $V\subset H \subset V^*$ denote a Gelfand triple, i.e. assume that $H$ is a separable Hilbert space, $V$ a reflexive Banach space and $V\subset H$ continuously and densely. Consider for  progressively measurable coefficients $A\colon [0,\infty)\times V \times \Omega \to V^*$, $B\colon [0,\infty)\times V  \times \Omega \to  \mathcal{L}_2(U,H)$ the equation
\begin{equation}\label{SPDE-LiuRoeckner1}
dX(t) = A(t,X(t))\dd t + B(t,X(t)) \dd W(t).
\end{equation}
We use the definition \cite[Definition 5.1.2]{LiuRoeckner} of a variational solution of \eqref{SPDE-LiuRoeckner1} on a deterministic bounded time interval $[0,T]$ and adapt it in the straightforward way to cover also local solutions: We say $(X, \sigma)$ is a \textit{local-in-time solution} of \eqref{SPDE-LiuRoeckner1} if there exists a non-decreasing sequence of $[0,\infty)$-valued stopping times $\sigma_n$ such that $\lim_{n\to\infty} \sigma_n = \sigma$ satisfying $\{\sigma_n<\sigma\}$ whenever $\{\sigma\not=0\}$ and $(X(t))_{t\in[0,\sigma)\cup\{0\}}$ is a continuous $H$-valued $(\F_t)$-adapted process satisfying
\begin{equation*}
 \hat{X}\one_{[0,\sigma_n]}\in L^\alpha([0,T]\times\Omega, \dd t\otimes\PP;V)\cap L^2([0,T]\times\Omega, \dd t \otimes \PP;H) \qquad \forall n\in\N,
\end{equation*}
where $\hat{X}$ defines the $dt \otimes \dd \PP$ equivalence class, and $\PP$-a.s.
\begin{equation*}
X(t\wedge\sigma_n) = X(0) + \int_0^{t\wedge\sigma_n} A(t,\bar X(s)) \dd s +
\int_0^{t\wedge\sigma_n}B(s,\bar X(s)) \dd W(s), t\in[0,T], \forall n\in\N,
\end{equation*}
where $\bar{X}$ is any $V$-valued progressively measurable $\dd t\otimes\dd \PP$ version of $\hat{X}$.

We say that local-in-time solutions are unique if for any two local-in-time solutions $(X^1, \sigma_1)$ and $(X^2, \sigma_2)$ with $X^1(0) = X^2(0)$ we have $X^1=X^2$ on $[0,\sigma_1 \wedge \sigma_2)$ $\PP$-almost surely.
We call a local-in-time solution $(X,\sigma)$ the \textit{unique maximal solution} of \eqref{SPDE-LiuRoeckner1} if, for any other local-in-time solution $(X^2, \sigma_2)$ of \eqref{SPDE-LiuRoeckner1}  with
the same initial condition, we have $\sigma^2 \leq \sigma$ and $X=X^2$ on $[0,\sigma^2)$ $\PP$-almost surely.

\begin{defi} \label{def:solHk} Assume \cref{assumption}. Consider the Gelfand triple $V:= H^{k+1}(\T^d)\subset  H:= H^k(\T^d) \subset H^{k+1}(\T^d)^*$ for $k$ as given in \cref{assumption} and $U=\ell^2$. 
\begin{enumerate}
\item We say $(u^\varepsilon,\sigma)$ is a variational solution / local-in-time solution / unique maximal solution of \eqref{eq:SPDEnonlocal} if for the coefficients
\begin{equation}
\begin{aligned}
\VV{A(u),\varphi} & := - \frac{1}{2}\HHH{ \sumi  g_\iota(u\falt) \Div(u g_\iota(u\falt))), \grad \varphi} \\
&\quad  - \frac{1}{2} \HHH{\sumi  u g_\iota'(u\falt) \Div(u g_\iota(u\falt))\falt), \grad \varphi}
 -  \HHH{\grad u, \grad \varphi}  \\
B(u)\psi & :=  -\sumi \Div(u g_\iota(u\falt)) \langle e_\iota, \psi\rangle_{\ell^2} 
\end{aligned}
\end{equation}
$(u^\varepsilon,\sigma)$ is a variational solution / local-in-time solution / unique maximal solution to \eqref{SPDE-LiuRoeckner1}.

\item We say $(u,\sigma)$ is a variational solution / local-in-time solution / unique maximal solution of \eqref{eq:SPDElocal} if for the coefficients
\begin{equation}
\begin{aligned}
\VV{A(u),\varphi} &:= -\frac{1}{2}\HHH{\sumi \g_\iota'(u)\Div(\g_\iota(u)), \grad \varphi}  -  \HHH{\grad u, \grad \varphi} \\
B(u)\psi &:= -\sumi  \Div(\g_\iota(u))\langle e_\iota, \psi\rangle_{\ell^2}
\end{aligned}
\end{equation}
$(u,\sigma)$ is a variational solution / local-in-time solution / unique maximal solution to \eqref{SPDE-LiuRoeckner1}.
\end{enumerate} 
\end{defi}

Before stating the main result, we briefly comment on the well-posedness properties of the SPDEs \eqref{eq:SPDElocal} and \eqref{eq:SPDEnonlocal} that will be used below and are established in \cref{sec:4}.
For the nonlocal-in-space equation \eqref{eq:SPDEnonlocal} with non-negative initial data, there exists a unique global-in-time solution. For the local-in-space equation \eqref{eq:SPDElocal}, we cannot establish under our general assumptions, without assuming additional smallness and decay assumptions on $g_\iota$,
that global-in-time solutions with $H^k(\T^d)$ regularity exist for all non-negative $u_0\in H^k(\T^d)$. However, we can establish that there exists a unique maximal solution $(u,\sigma)$ to \eqref{eq:SPDElocal} with paths in $C([0,\sigma),H^k(\T^d))$.

In addition, we characterize what it means for the paths of solutions to \eqref{eq:SPDElocal} to have 
$H^k(\T^d)$ regularity: We show a blow up criterion, that can informally be seen as stating the following. Assume that a solution $u(t)$ (of some suitable type) satisfies up to some stopping time $\tau$
\[\int_0^\tau \|u(t)\|_{W^{1,\infty}}^2 \mathrm dt<\infty \quad \PP \text{-a.s.} \qquad \text{or} \qquad 
\sup_{t\leq \tau} \|u(t)\|_{C^{\alpha}(\T^d)} <\infty \quad \PP \text{-a.s. for some } \alpha>0
 \]
then
\[
\sup_{t\leq \tau} \|u(t)\|_{H^k(\T^d)}<\infty \qquad \PP \text{-almost surely}
\]
hence preventing blow-up of the maximal solution before time $\tau$. In particular, this allows us to control the high Sobolev regularity required in the main result through the time-integrated $W^{1,\infty}$-norm rather than an a priori bound on the full $H^k$-norm. These statements are established in \cref{thm:local-in-spaceEquationWP,thm:nonlocal-in-spaceEquationWP} below.

\begin{theorem}[Main theorem]\label{thm:main}
Assume \cref{assumption}, $k$ as given by \cref{assumption},\\ $u_0\in L^2(\Omega, \F_0, \PP; H^{k+1}(\T^d))$, $u_0 \geq 0$ and $T>0$. Let $(u,\sigma)$ denote the unique maximal solution of \eqref{eq:SPDElocal} with localizing sequence $(\tilde{\sigma}_n)_{n\in\N}$ and $u(0)=u_0$. Moreover, for every $\varepsilon\in(0,1)$, let $(u^\varepsilon(t))_{t\in[0,\infty)}$ denote the unique global maximal solution of \eqref{eq:SPDEnonlocal} with $u^\varepsilon(0)=u_0$. Then, we have for every $m\in\N$ that
\begin{equation*}
 \lim_{\varepsilon\to 0} \sup_{t\in[0,\tilde{\sigma}_m \wedge T]} \|u^\varepsilon(t) - u(t)\|_{H^k(\T^d)} = 0  \qquad \text{ in probability.}
\end{equation*}
\end{theorem}


\section{Tamed SPDEs}
\label{sec:3}

In this section we introduce tamed versions of the SPDEs  \eqref{eq:SPDEnonlocal} and \eqref{eq:SPDElocal}. We apply the tamed versions of the SPDEs to avoid technical complications in the proof of
 \cref{thm:main} in \cref{sec:5,sec:6}, in which we would like to stop the processes $(\ue(t))_{t\in[0,\infty)}$  at the time $$\tau_{\varepsilon} := \inf\{t>0 \mid \|\ue(t)\|_{H^k(\T^d)} > R\},$$ where $R>0$ is a fixed parameter. By using the tamed SPDEs we can avoid these $\varepsilon$-dependent stopping times. Moreover, we also use them in Section 4 to prove the existence of unique maximal solutions.

\subsection{Tamed SPDEs: Definition}
Let $\gamma_R$ be as defined in \cref{assumption}. We tame the nonlocal-in-space SPDE \eqref{eq:SPDEnonlocal} as follows:
\begin{equation}\label{eq:approx-cut-nonlocal}\tag{$tNL_\varepsilon$}
\begin{aligned}
\dd \uer &  = \frac{1}{2}  \sum_{\iota \geq 1} \gamma^2_R( \|\uer\|_{H^{k}}) \Div(g_{\iota} \Div(\uer g_{\iota}))\dd t \\
& \quad  + \frac{1}{2}   \sum_{\iota \geq 1} \gamma^2_R( \|\uer \|_{H^{k}}) \Div(\uer g'_{\iota} \Div(\uer g_{\iota})\ast\eta_\varepsilon)\dd t \\
& \quad  +  \Delta \uer \dd t -  \sum_{\iota \geq 1} \gamma_R(\|\uer \|_{H^{k}})  \Div(\uer g_{\iota}) \dd W^\iota_t
\end{aligned}
\end{equation}
where we shorten the notation by $g_\iota:= g_\iota(\uer \ast\eta_\varepsilon, x)$  and $g'_\iota:= g'_\iota(\uer \ast\eta_\varepsilon, x)$. Likewise, we tame the local-in-space SPDE \eqref{eq:SPDElocal} as follows
\begin{equation}\label{eq:approx-cut-local}\tag{$tL_0$}
\begin{aligned}
\dd \ur = \frac{1}{2}  \sum_{\iota \geq 1} \gamma^2_R( \|\ur\|_{H^{k}}) \Div(\g_{\iota}' \Div(\g_{\iota}))\dd t + \Delta \ur \dd t -  \sum_{\iota \geq 1}  \gamma_R(\|\ur\|_{H^{k}})  \Div(\g_{\iota}) \dd W^\iota_t,
\end{aligned}
\end{equation}
where we shorten the notation by $\g_\iota:= \ur g_\iota(\ur, x)$ and $\g_\iota':= g_\iota(\ur, x) + \ur g_\iota'(\ur, x)$. \\

\begin{defi}\label{def:approx-cut-sol} Assume \cref{assumption}. Consider the Gelfand triple $V:= H^{k+1}(\T^d)\subset  H:= H^k(\T^d) \subset H^{k+1}(\T^d)^*$ for $k$ as given in \cref{assumption}. 
\begin{enumerate}
\item We say $\uer$ is a variational solution / local-in-time solution / unique maximal solution of \eqref{eq:approx-cut-nonlocal} if for the coefficients  
\begin{equation}
\begin{aligned}
\VV{A(u),\varphi} & := - \frac{1}{2} \HHH{ \sumi  \gamma^2_R( \|u\|_{H^{k}})  g_\iota(u\falt) \Div(u g_\iota(u\falt))), \grad \varphi} \\
&\quad  - \frac{1}{2} \HHH{\sumi  \gamma^2_R( \|u\|_{H^{k}}) u g_\iota'(u\falt) \Div(u g_\iota(u\falt))\falt), \grad \varphi}
 -  \HHH{\grad u, \grad \varphi}  \\
B(u)\psi & :=  -\sumi \gamma_R( \|u\|_{H^{k}}) \Div(u g_\iota(u\falt)) \langle e_\iota, \psi\rangle_{\ell^2} 
\end{aligned}
\end{equation}
$\uer$ is a variational solution / local-in-time solution / unique maximal solution to \eqref{SPDE-LiuRoeckner1}.

\item We say $\ur$ is a variational solution / local-in-time solution / unique maximal solution of \eqref{eq:approx-cut-local} if for the coefficients  
\begin{equation}
\begin{aligned}
\VV{A(u),\varphi} &:= -\frac{1}{2}\HHH{\sumi \gamma^2_R( \|u\|_{H^{k}}) \g_\iota'(u)\Div(\g_\iota(u)), \grad \varphi}  -  \HHH{\grad u, \grad \varphi} \\
B(u)\psi &:= -\sumi \gamma_R( \|u\|_{H^{k}}) \Div(\g_\iota(u))\langle e_\iota, \psi\rangle_{\ell^2}
\end{aligned}
\end{equation}
$\ur$ is a variational solution / local-in-time solution / unique maximal solution to \eqref{SPDE-LiuRoeckner1}.
\end{enumerate} 
\end{defi}

\subsection{SPDEs with locally monotone coefficients}

There are different options to establish the well-posedness theory for for our equations  \eqref{eq:SPDEnonlocal}, \eqref{eq:approx-cut-nonlocal},  \eqref{eq:SPDElocal} and  \eqref{eq:approx-cut-local}. One option is the critical variational setting, see for example Agresti and Veraar \cite[Theorem 3.3, Theorem 3.4, Theorem 3.5]{AgrestiVeraar2024}. Another option is to apply the framework of Liu and Röckner, see \cite{LiuRoeckner}. In this paper we apply the latter.

To establish well-posedness for the tamed SPDEs, we observe that neither  \cite[Theorem 4.2.4]{LiuRoeckner}  nor \cite[Theorem 5.1.3]{LiuRoeckner} applies directly. More precisely, (H2) of Theorem 4.2.4 is not met as our equation is only locally monotone. On the other hand, the noise term contains a gradient, so the assumption 
\begin{equation}\label{eq:noiseAssumption}
\|B(t,v)\|_{L_2(U,H)}^2 \leq C(f(t) + \|v\|_H^2)
\end{equation}
 of Theorem 5.1.3 is not satisfied. However, our tamed SPDEs satisfy the stronger growth condition (H4) of Theorem 4.2.4, which  ensures well-posedness by a minor modification of the original proof of Theorem 5.1.3. The following corollary states the mixture of \cite[Theorem 4.2.4]{LiuRoeckner} and \cite[Theorem 5.1.3]{LiuRoeckner} our tamed SPDEs require. Gradient-dependent noise is also treated in \cite[Theorem 3.2]{RocknerShangZhang2024} and \cite{RocknerShangZhang2026Correction}, under a smallness condition not imposed here.

\begin{corollary}[] \label{cor:LiuRoeckner}
Suppose there exist constants $\alpha \in ]1,\infty[$, $\theta \in ]0,\infty[$, $C_0 \in \R$ and a non-negative adapted process $f\in L^1([0,T]\times\Omega, \dd t \otimes \PP)$ such that the following conditions hold for all $u, v, w \in V$ and $(t,\omega) \in [0,T]\times \Omega$:
\begin{enumerate}
 \item[(H1)] (Hemicontinuity) The map $\lambda \mapsto \leftindex_{V^*}{\langle}A(t,u+\lambda v), w\rangle_V$ is continuous on $\R$.
 \item[(H2')] (Local monotonicity)
 \begin{equation*}
\begin{aligned}
& 2 \leftindex_{V^*}{\langle} A(t,u) - A(t,v), u-v\rangle_V + \|B(t,u)-B(t,v)\|^2_{\LUH} \leq (f(t)  + \rho(u))||u-v||_H^2
\end{aligned}
\end{equation*}
 where $\rho:V\to[0,\infty[$ is a measurable hemicontinuous function and locally bounded in $V$.
 \item[(H3)] (Coercivity)
 \begin{equation*}
 2 \leftindex_{V^*}{\langle} A(t,v),v\rangle_V + \|B(t,v)\|^2_{\LUH} \leq C_0 \|v\|_H^2 - \theta \|v\|_V^\alpha + f(t).
\end{equation*}
 \item[(H4)] (Growth)
  \begin{equation*}
 \|A(t,v)\|_{V^*}^{\alpha/(\alpha - 1)} \leq f(t) + C_0 \|v\|^\alpha_V.
\end{equation*}
\end{enumerate}
Moreover, assume there exists a constant C such that
\begin{equation}\label{eq:LR-samecondition}
 \rho(v)\leq C(1+\|v\|_V^\alpha) \qquad \forall v\in V.
\end{equation}

Then for any $X_0\in L^2(\Omega, \F_0, \PP; H)$, \eqref{SPDE-LiuRoeckner1} has a unique variational solution $(X(t))_{t\in[0,T]}$ (in the sense of \cite[Definition 5.1.2]{LiuRoeckner}) such that $X(0)=X_0$ and it satisfies
$\E[\sup_{t\in[0,T]}\|X(t)\|_H^2] < \infty$. Moreover, local-in-time solutions are also unique.
\end{corollary}

\begin{proof}[Proof of \cref{cor:LiuRoeckner}]  
In this proof we use the notation of \cite[Theorem 5.1.3]{LiuRoeckner}.

The uniqueness of local-in-time solutions follows from obvious minor modifications of the uniqueness proof of \cite[Theorem 5.1.3]{LiuRoeckner}.

The existence of a unique variational solution follows from the proof of \cite[Theorem 5.1.3]{LiuRoeckner} (for $p=2$) with one minor modification: The assumption \eqref{eq:noiseAssumption} was only used in the proof of (5.4), \cite[Lemma 5.1.5.]{LiuRoeckner}. Hence, it suffices to show for the Galerkin approximations $X^{(n)}$ that there exists a $C>0$ such that for all $n\in\N$ we have
\begin{equation}\label{eq:cor32SufficesToShow}
\E\left[ \sup_{t\in[0,T]}  \|X^{(n)}(t)\|_H^2 \right]
+ \E\left[\int_0^T \|X^{(n)}(t)\|_V^\alpha \dd t\right] \leq 
C\left( \E[\|X^{(n)}(0)\|_H^2] +\E\left[\int_0^T f(t) \dd t\right]\right).
\end{equation}
To this end, let $\tau_R^{(n)}$, $R\in\N$ denote a localizing sequence for $X^{(n)}$. By Ito's formula for $\R^n$-valued semimartingales, (H3), discarding the negative term $- \theta \|X^{(n)}(s)\|_V^\alpha$, the Burkholder-Davis-Gundy inequality  and Young's inequality we have for any $\varepsilon>0$
\begin{equation*}
\begin{aligned}
 \E\left[ \sup_{s\in[0,\trn]} \|X^{(n)}(s)\|_H^2 \right]
 & \leq \E[\|X^{(n)}(0)\|_H^2] + \E\left[\int_0^\trn ( C_0 \|X^{(n)}(s)\|_H^2  + |f(s)|) \dd s\right]  \\
&\quad + 3\varepsilon \E\left[ \sup_{s\in[0,\trn]} \|X^{(n)}(s)\|_H^2 \right] + 3  \varepsilon^{-1} \E\left[\int_0^\trn \|B(s,X^{(n)}(s))\|^2_{\LUH}   \dd s\right]. \\
\end{aligned}
\end{equation*}
Using that by (H3) and (H4) we have
\begin{equation*}
\begin{aligned}
\|B(t,v)\|_{\LUH}^2  
& \leq 
C_0 \|v\|^2_H    
+ \left(1 + 2 \frac{\alpha-1}{\alpha} \right) f(t) 
+ \left(\frac{2}{\alpha} +  2C_0 \frac{\alpha-1}{\alpha} - \theta\right)   \|v\|_V^\alpha,
\end{aligned}
\end{equation*}
noting that  $\E\left[\int_0^T \|X^{(n)}(t)\|_V^\alpha \dd t\right]$ is bounded due to (H3), choosing $\varepsilon>0$ sufficiently small, rearranging the terms, applying the deterministic Gronwall inequality and taking the limit $R\to\infty$ implies \eqref{eq:cor32SufficesToShow}.
\end{proof}

\subsection{Tamed SPDEs: Well-posedness}
In this subsection, we establish well-posedness of the tamed SPDEs by applying \cref{cor:LiuRoeckner}. We first derive auxiliary estimates for the coefficients of the non-local SPDE, collected in \cref{lemma:BoringEstimate}, which will be reused later in the paper. Well-posedness of the local and non-local tamed SPDEs is then obtained in \cref{thm:local-in-spaceEquationCutoffWP,thm:nonlocal-in-spaceEquationCutoffWP}, respectively.

\begin{lemma}\label{lemma:BoringEstimate}
Assume \cref{assumption}, $R\in(0,\infty]$, $\varepsilon\in(0,1)$ and $(H,V)=(H^k(\T^d),H^{k+1}(\T^d))$ or $(H,V)=$ \\ $(H^{k+1}(\T^d),H^{k+2}(\T^d))$. The coefficients $A$ and $B$ defined in  \cref{def:approx-cut-sol} 1) satisfy for all $v\in V$
 \begin{equation*}
 \begin{aligned}
 & 2 \leftindex_{V^*}{\langle} A(t,v),v\rangle_V + \|B(t,v)\|^2_{\LUH}  \\
 & \qquad \leq   C_{d,\varepsilon,1}(1+ C_g^2)  \gamma_R^2(\|v\|_{H^k}) (1+\|v\|_{L^1(\T^d)})^{2(k+2)} 
  \|v\|_H^2  + \|v\|_H^2 - \|v\|_V^2 
\end{aligned}
\end{equation*}
and
\begin{equation*}
 \|A(t,v)\|_{V^*}^2 \leq   C_{d, \varepsilon,2} \left(\sumi \|g_\iota\|_{C^{k+2}(M_\varepsilon)}^2 \gamma_R^2(\|v\|_{H^k}) (1+\|v\|_{L^1(\T^d)})^{k+2} \right)^2 \|v\|^2_V + C_d\|v\|_V^2.
\end{equation*}
Here, 
\begin{equation*}
C_g:=  \sumi \|g_\iota\|_{C^{k+2}(M_\varepsilon)}^2, \qquad  M_\varepsilon :=  [-\|v\ast\eta_{\varepsilon}\|_{L^\infty},\|v\ast\eta_{\varepsilon}\|_{L^\infty}]\times\T^d.
\end{equation*}
The constants $C_{d,\varepsilon,1}$ and $C_{d,\varepsilon,2}$ depend on $d$, $\varepsilon$ and $\eta$ and are independent of $R$; $C_d$ depends only on $d$.
\end{lemma}

\begin{proof}[Proof of \cref{lemma:BoringEstimate}]
We prove the coercivity estimate; the growth estimate follows by similar calculations.  For any $v\in V$ we have by definition of the coefficients $A$, $B$:
\begin{equation*}
\begin{aligned}
&  2 \leftindex_{V^*}{\langle} A(t,v),v\rangle_V + \|B(t,v)\|^2_{\LUH} \\
& \quad = - 2\|\grad v\|^2_H  -\langle \sum_{\iota \geq 1} \gamma^2(v) \gv \Div(v\gv )), \grad v\rangle_H  \\
& \qquad -\langle \sum_{\iota \geq 1}  \gamma^2(v) v\gdv  \Div(v\gv)\ast\eta_\varepsilon, \grad v\rangle_H
+ \sum_{\iota \geq 1}  \|\gamma(v) \Div(v g_\iota(v\falt))\|_H^2\\
\end{aligned}
\end{equation*}
Let $\alpha\in\N^d$ such that $|\alpha| = j \leq k+1$. We estimate the lower order terms by noting $$\|\partial^{j}_{\alpha} \gv \|_{L^\infty(\T^d)} \leq (1+\|v\|_{L^1(\T^d)})^{j} \|g_\iota\|_{C^{j}(M_\varepsilon)}C_\varepsilon, $$ where 
$M_\varepsilon:= [-\|v\ast\eta_{\varepsilon}\|_{L^\infty},\|v\ast\eta_{\varepsilon}\|_{L^\infty}]\times\T^d$ and $C_\varepsilon$ denotes a constant that depends on $\varepsilon$ and $\eta$. We obtain for any $\delta>0$
\begin{equation*}
\begin{aligned}
& -\langle  \partial_\alpha^j \sum_{\iota \geq 1} \gamma^2(v) \gv \Div(v\gv )), \partial_\alpha^j \grad v\rangle_{L^2(\T^d)}  \\
& -\langle \partial_\alpha^j  \sum_{\iota \geq 1}  \gamma^2(v) v\gdv  \Div(v\gv)\ast\eta_\varepsilon, \partial_\alpha^j \grad v\rangle_{L^2(\T^d)} \\
& + \sum_{\iota \geq 1}  \| \partial_\alpha^j (\gamma(v) \Div(v g_\iota(v\falt)))\|^2_{L^2(\T^d)}\\
&\quad  \leq -\langle  \sum_{\iota \geq 1} \gamma^2(v) \gv (\partial_\alpha^j \grad v)\gv ), \partial_\alpha^j \grad v\rangle_{L^2(\T^d)} \\
&\qquad + \sum_{\iota \geq 1}  \langle  \gamma(v) (\partial_\alpha^j \grad v) g_\iota(v\falt),  \gamma(v) (\partial_\alpha^j \grad v) g_\iota(v\falt)    )\rangle_{L^2(\T^d)}\\
& \qquad  + \delta \| \partial_\alpha^j \grad v\|_{L^2(\T^d)}^2 + \frac{1}{4\delta}\left(C_{\alpha,1} \sumi \|g_\iota\|_{C^{k+2}(M_\varepsilon)}^2 \gamma^2(v) (1+\|v\|_{L^1(\T^d)})^{k+2} C_\varepsilon^2 \right)^2  \|v\|_H^2 \\
& \qquad + \left(C_{\alpha,2}\sumi \|g_\iota\|_{C^{k+2}(M_\varepsilon)}^2 \gamma^2(v) (1+\|v\|_{L^1(\T^d)})^{2(k+2)} C_\varepsilon^2 \right)  \|v\|_H^2 
\end{aligned}
\end{equation*}
Taking the Laplacian term into account, summing over $|\alpha| \leq k$ or  $|\alpha| \leq k+1$, and choosing $\delta >0$ sufficiently small implies the assertion.

\end{proof}

\begin{lemma}\label{thm:local-in-spaceEquationCutoffWP}
Assume \cref{assumption}. Then, for all $u_0\in  L^2(\Omega, \F_0, \PP; H^k(\T^d))$, $T>0$ and $R\in(0,\infty)$ there exists a unique variational solution $(u^R(t))_{t\in[0,T]}$ of \eqref{eq:approx-cut-local}. Moreover, local-in-time solutions are unique.
\end{lemma}

\begin{proof}[Proof of \cref{thm:local-in-spaceEquationCutoffWP}] The assertion of the lemma follows immediately from an application of \cref{cor:LiuRoeckner} by choosing $H = H^{k}(\T^d)$ and $V=H^{k+1}(\T^d)$. 
The straightforward but somewhat lengthy calculations verifying that the assumptions (H1), (H2'), (H3) and (H4) are indeed satisfied are omitted here.
\end{proof}

\begin{lemma}\label{thm:nonlocal-in-spaceEquationCutoffWP}
Assume \cref{assumption}. Then, for all $u_0\in  L^2(\Omega, \F_0, \PP; H^k(\T^d))$, $T>0$, $R\in(0,\infty)$ and $\varepsilon\in(0,1)$ there exists a unique variational solution $(\uer(t))_{t\in[0,T]}$ of \eqref{eq:approx-cut-nonlocal}.
If in addition $u_0\in  L^2(\Omega, \F_0, \PP; H^{k+1}(\T^d))$, then $(\uer(t))_{t\in[0,T]}$ is a continuous $H^{k+1}(\T^d)$-valued $(\F_t)$-adapted process and satisfies $\hat{u}^{\varepsilon,R} \in L^2([0,T]\times\Omega, \dd t \otimes \dd \PP, H^{k+2}(\T^d))$. Moreover, local-in-time solutions are unique.
\end{lemma}

\begin{proof}[Proof of \cref{thm:nonlocal-in-spaceEquationCutoffWP}] 

\textit{Existence of a variational solution:} The existence of a variational solution follows from an application of \cref{cor:LiuRoeckner} for the choice $H = H^{k}(\T^d)$ and $V=H^{k+1}(\T^d)$. Assumptions (H3) and (H4) of  \cref{cor:LiuRoeckner} follow from 
\cref{lemma:BoringEstimate}. The straightforward but somewhat lengthy calculations to verify the assumptions (H1) and (H2') are omitted here.

\textit{Additional regularity of the variational solution:} As all assumptions of \cref{cor:LiuRoeckner} for the choice $H = H^{k}(\T^d)$ and $V=H^{k+1}(\T^d)$ are satisfied, this implies in particular also that the Galerkin approximations $X^{(n)}$  as defined in  \cite[Theorem 5.1.3.]{LiuRoeckner} converge weakly in $K:=L^{2}([0,T]\times\Omega, \dd t \otimes \dd \PP, H^{k+1}(\T^d))$ to the solution $\uer$ of the SPDE. In this construction, we may choose the orthonormal basis of
$H^k(\T^d)$ used in the construction of the Galerkin approximations to be a complete orthogonal system of $H^{k+1}(\T^d)$ and  to satisfy $(e_k)_{k\in\N} \subset C^\infty(\T^d)$. For $u\in H_n:= \textrm{span}\{e_1, \cdots, e_n\}\subset C^\infty(\T^d)$ define
\begin{equation*}
\begin{aligned}
\bar{A}(t,u) & := \frac{1}{2} \sumi  \gamma^2_R( \|u\|_{H^{k}})  \Div(g_\iota(u\falt) \Div(u g_\iota(u\falt)))) \\
& + \frac{1}{2} \sumi  \gamma^2_R( \|u\|_{H^{k}}) \Div(u g_\iota'(u\falt) \Div(u g_\iota(u\falt))\falt)) + \Delta u. 
\end{aligned}
\end{equation*}
Due to the regularity of $u\in C^\infty(\T^d)$, we have $\bar{A}(t,u) \in H^{k+1}(\T^d)$ and hence $\leftindex_{H^l(\T^d)^*}{\langle} A(t,u), v\rangle_{H^l(\T^d)} = \langle \bar{A}(t,u), v\rangle_{H^l(\T^d)}$ for $l=k, k+1$. 
By \cref{lemma:BoringEstimate}, conditions (H3) and (H4) are satisfied for $(H,V)=(H^{k+1}(\T^d),H^{k+2}(\T^d))$. In particular, (H3) implies that $X^{(n)}$ are bounded in $L^{2}([0,T]\times\Omega, \dd t \otimes \dd \PP, H^{k+2}(\T^d))$ and hence, by weak convergence,
\begin{equation}\label{lemma34:extrareg}
\hat{u}^{\varepsilon, R} \in L^{2}([0,T]\times\Omega, \dd t \otimes \dd \PP, H^{k+2}(\T^d)).
\end{equation}
By \eqref{lemma34:extrareg}, (H3) and (H4) for $(H, V) = (H^{k+1}(\T^d),  H^{k+2}(\T^d))$ and \cite[Theorem 4.2.5]{LiuRoeckner} we may conclude that $(u^{\varepsilon,R}(t))_{t\in[0,T]}$ is a continuous $H^{k+1}(\T^d)$-valued adapted process.
\end{proof}

\section{Regularity theory for the local and nonlocal SPDEs}\label{sec:4}
In this section, we establish regularity theory for the local-in-space and nonlocal-in-space SPDEs. More precisely, we prove local well-posedness and a blow-up criterion for the local-in-space equation (see \cref{thm:local-in-spaceEquationWP}), as well as global well-posedness for the nonlocal-in-space equation (see \cref{thm:nonlocal-in-spaceEquationWP}). The local existence results themselves follow directly from the global well-posedness of the corresponding tamed SPDEs established in the previous section. The main point of the present section is therefore the derivation of the blow-up criterion for the local-in-space equation.

This blow-up criterion can be informally interpreted as follows: It characterizes what it means for solutions to \eqref{eq:SPDElocal} to have 
$H^k(\T^d)$ regularity for $k=\ceil{d/2 + 3/2}$ (and it would similarly also work for higher $k$). Assume that $u(t)$ is a solution (of some suitable type) to the local-in-space SPDE that satisfies up to some stopping time $\tau$
\[\int_0^\tau \|u(t)\|_{W^{1,\infty}}^2 \mathrm dt<\infty \quad \PP \text{-a.s.} \qquad \text{or} \qquad 
\sup_{t\leq \tau} \|u(t)\|_{C^{\alpha}(\T^d)} <\infty \quad \PP \text{-a.s. for some } \alpha>0,
 \]
then
\[
\sup_{t\leq \tau} \|u(t)\|_{H^k(\T^d)}<\infty \qquad \PP \text{-almost surely}
\]
hence preventing blow-up of $H^k(\T^d)$ norm before time $\tau$. In particular, this allows us to control the high Sobolev regularity required in the main result through the time-integrated $W^{1,\infty}$-norm rather than an a priori bound on the full $H^k$-norm.

\subsection{Local SPDE: Local well-posedness and blow-up criterion}
\label{subsec:thm:nonlocal-in-spaceEquationWP}

The main result of this subsection is \cref{thm:local-in-spaceEquationWP}, which establishes the existence of a unique maximal solution to the local SPDE \eqref{eq:SPDElocal} and provides the blow-up criterion. We first note two simple consequences of the Gagliardo–Nirenberg inequality (\cref{rmk:simpleInequality}), which is then used to establish energy estimates (\cref{lemma:ItoLocalEq}). Combining these estimates with a stochastic version of Gronwall's inequality (\cref{lemma:StochGron}) yields the proof of  \cref{thm:local-in-spaceEquationWP}.

\begin{remark}\label{rmk:simpleInequality} Let $v\in H^{k+1}(\T^d)$ for $k\geq  \lceil d/2 + 3/2\rceil$, $n\geq 2$, $1\leq l_i \leq k$, $\sum_{i=1}^n l_i \leq k+1$, $\beta_i \in \N^d$, $|\beta_i| = l_i$ and set 
\begin{equation*}
P(v) := \prod_{i=1}^n \partial^{l_i}_{\beta_i} v.
\end{equation*}
 Choose $\theta$ such that $0<\theta<1/k$ and let $X_\theta$ be a fractional Sobolev (Slobodeskii) space $W^{s,p}(\T^d)$, $p\in[1,\infty), s\in(0,k]$, with Sobolev index $\theta:= s-d/p$.
Then, for every $\delta>0$ there exists a constant $C_{\theta,\delta} = C(k, d, \theta, s, p, \delta)>0$ such that $P(v)$
can be estimated by 
\begin{equation}\label{eq:Pv-estimate}
\|v\|_{H^{k+1}} \|P(v)\|_{L^2} + \|P(v)\|_{L^2}^2 \leq \delta \|v\|_{H^{k+1}}^2 + C_{\theta,\delta}\|v\|_{X_\theta}^{2/\theta}\|v\|_{H^k}^2.
\end{equation}
Moreover, for every $\delta >0$, there exists a $C_\delta=C(\delta, d,k)>0$ such that
\begin{equation}\label{eq:Pv-estimate2}
\|v\|_{H^{k+1}}\|P(v)\|_{L^2}
+\|P(v)\|_{L^2}^2
\le
\delta\|v\|_{H^{k+1}}^2
+C_\delta\bigl(1+\|v\|_{L^\infty}^{2(k-1)}\bigr)
\|v\|_{W^{1,\infty}}^2
\|v\|_{H^k}^2.
\end{equation}

Inequality \eqref{eq:Pv-estimate} can be verified as follows: Define
\begin{equation*}
k_\theta := k+1 - (n-1)\theta \in (k,k+1), \quad \lambda_i := \frac{l_i - \theta}{k_\theta - \theta},
\end{equation*}
noting $\sum_{i=1}^n \lambda_i \leq 1$. Then, applying Hölder's inequality in the first inequality, a Gagliardo-Nirenberg inequality for fractional Sobolev spaces \cite[Theorem 1]{BrezisMironescu2019} in the second inequality ensured by $p<\infty$, and $W^{k_\theta,2}\hookrightarrow X_\theta$ in the third, and e.g. \cite[Theorem 1]{BrezisMironescu2018} in the fourth inequality, we have
\begin{equation*}
\begin{aligned}
\|P(v)\|_{L^2} \lesssim \prod_{i=1}^n \|\partial^{l_i}_{\beta_i} v\|_{2/\lambda_i} \lesssim \prod_{i=1}^n \|v\|_{X_\theta}^{1-\lambda_i}\|v\|_{W^{k_\theta,2}}^{\lambda_i}
\lesssim \|v\|_{X_\theta}^{n-1}\|v\|_{W^{k_\theta,2}} \lesssim \|v\|_{X_\theta}^{n-1}\|v\|^{(n-1)\theta}_{H^{k}} \|v\|^{1-(n-1)\theta}_{H^{k+1}} .
\end{aligned}
\end{equation*}
Now \eqref{eq:Pv-estimate} follows from an application of Young's inequality.

For inequality \eqref{eq:Pv-estimate2}, set $m:=\sum_{i=1}^{n}l_i-1\le k$.
First suppose that \(m\ge 2\), and choose $\eta:=\frac{m-1}{m(n-1)-1}\in(0,1]$. H\"older's inequality, applied after splitting each factor into
powers \(1-\eta\) and \(\eta\), gives
\begin{equation}
\label{eq:product-holder}
\begin{aligned}
\|P(v)\|_{L^2}
&=
\left\|
\prod_{i=1}^{n}|\partial^{l_i}_{\beta_i}v|^{1-\eta}
\prod_{i=1}^{n}|\partial^{l_i}_{\beta_i}v|^\eta
\right\|_{L^2}\\
&\le
\left(
\prod_{i=1}^{n}
\|\partial^{l_i}_{\beta_i}v\|_{L^{2m/l_i}}
\right)^{1-\eta}
\left(
\prod_{i=1}^{n}
\|\partial^{l_i}_{\beta_i}v\|_{L^{2(m-1)/(l_i-1)}}
\right)^\eta.
\end{aligned}
\end{equation}
Indeed, the H\"older exponents satisfy
\begin{equation*}
(1-\eta)\sum_{i=1}^{n}\frac{l_i}{2m}
+\eta\sum_{i=1}^{n}\frac{l_i-1}{2(m-1)}
=\frac12.
\end{equation*}
Here we used the convention $2(m-1)/(l_i-1)=\infty$ for $l_i=1$, and factors raised to the power zero are omitted.
We estimate \eqref{eq:product-holder} by applying the following two
Gagliardo--Nirenberg inequalities
\begin{equation}
\label{eq:gn-v}
\|\partial^{l_i}_{\beta_i}v\|_{L^{2m/l_i}}
\lesssim \|v\|_{L^\infty}^{\,1-l_i/m}
 \|v\|_{H^m}^{\,l_i/m},
\end{equation}
and
\begin{equation}
\label{eq:gn-gradient}
\|\partial^{l_i}_{\beta_i}v\|_{L^{2(m-1)/(l_i-1)}}
\lesssim \|\nabla v\|_{L^\infty}^{\,1-(l_i-1)/(m-1)}
 \|v\|_{H^m}^{\,(l_i-1)/(m-1)}.
\end{equation}
These follow by applying the classical Gagliardo--Nirenberg
inequality for bounded domains to $v$ and its first derivatives, respectively. Here, we note that there are two endpoint cases that can be handled directly without the Gagliardo--Nirenberg inequality: If $l_i = 1$, then \eqref{eq:gn-gradient} is immediate. If $l_i = m$, then $n=2$ and the other derivative must be of order $1$. Hence, in this special case, \eqref{eq:Pv-estimate2} is immediate.
Substituting \eqref{eq:gn-v} and \eqref{eq:gn-gradient} into \eqref{eq:product-holder} implies by $m\leq k$
\begin{equation*}
\|P(v)\|_{L^2}
\lesssim \|v\|_{L^\infty}^{n-2}\|\grad v\|_{L^\infty} \|v\|_{H^m}
\lesssim \|v\|_{L^\infty}^{n-2}\|\grad v\|_{L^\infty} \|v\|_{H^k},
\end{equation*}
which in turn implies by $n\leq k+1$ and Young's inequality \eqref{eq:Pv-estimate2} for $m\geq 2$. If $m=1$, then $n=2$ and $l_1=l_2=1$, so $\eqref{eq:Pv-estimate2}$ can be easily seen.
\end{remark}

\begin{lemma}\label{lemma:ItoLocalEq}
Assume \cref{assumption}, $u_0\in  L^2(\Omega, \F_0, \PP; H^k(\T^d))$, that $(u,\sigma)$ is a local-in-time solution of \eqref{eq:SPDElocal} and let $k$ be given as in \cref{assumption}. Let $(\tilde{\sigma}_n)_n$ denote the localizing sequence of $(u,\sigma)$. Define for all $n\in\N$
\begin{equation*}
\sigma_n := \inf\{t\geq 0 \mid \|u(t)\|_{H^k} \geq n\}\wedge\tilde{\sigma}_n.
\end{equation*}
\begin{enumerate}
\item  Choose $\theta$ such that $0<\theta<1/k$ and let $X_\theta$ be a fractional Sobolev space $W^{s,p}(\T^d)$, $p\in[1,\infty), s\in(0,k]$, with Sobolev index $\theta$.
Then there exists a continuous local martingale $(M_t)_{t\in[0,\sigma)}$  with $M_0=0$ such that, for $t\geq 0$, $n\in\N$,
\begin{equation*}
\begin{aligned}
\|u(t\wedge\sigma_n)\|_{H^k}^2+ \inttn \|\grad u\|_{H^k}^2 \dd t
 &\leq\|u_0\|_{H^k}^2 + C_{\theta,g}\inttn \big(1+ \|u\|_{H^k}^2\big)\big(1+\|u\|_{X_\theta}^{2/\theta}\big) \dd t + M_{t\wedge\sigma_n},
\end{aligned}
\end{equation*}
where $C_{\theta, g}\in L^0(\Omega, \F_0, \PP)$ depends on $\theta$, $s$, $p$, $(g_{\iota})_{\iota \geq 1}$, $d$ and $\|u_0\|_{L^\infty}$.
\item There exists a continuous local martingale $(M_t)_{t\in[0,\sigma)}$  with $M_0=0$ such that, for $t\geq 0$, $n\in\N$,
\begin{equation*}
\begin{aligned}
\|u(t\wedge\sigma_n)\|_{H^k}^2+ \inttn \|\grad u\|_{H^k}^2 \dd t
 &\leq\|u_0\|_{H^k}^2 + C_{g}\inttn \big(1+ \|u\|_{H^k}^2\big)\big(1+\|u\|_{W^{1,\infty}}^2\big) \dd t + M_{t\wedge\sigma_n},
\end{aligned}
\end{equation*}
where $C_{g}\in L^0(\Omega, \F_0, \PP)$ depends on $(g_{\iota})_{\iota \geq 1}$, $d$ and $\|u_0\|_{L^\infty}$.
\end{enumerate}

\end{lemma}

\begin{proof}[Proof of \cref{lemma:ItoLocalEq}]
By replacing $u_0$ by $u_0 \one_{\{\|u_0\|_{H^k}<N\}}$, $N\in\N$ if necessary, we may assume without loss of generality that  $u_0\in  L^\infty(\Omega, \F_0, \PP, H^k(\T^d))$.

We have 
\begin{equation}\label{eq:Lemma55-proof-1}
\sup_{t<\sigma} \|u(t)\|_\infty \leq \|u_0\|_\infty \quad \PP-a.s.
\end{equation}
which can be seen as follows: For any convex $\varphi\in C^2(\R,\R)$ with bounded second derivative and any $n\in\N$, $t\geq 0$ we have by It\^o's formula  \cite[Proposition A.1]{Ito-DebusscheHofmanovaVovelle}
\begin{equation*}
\intT \varphi(u(t\wedge\sigma_n))  \leq  \intT \varphi(u_0)  + M_{t\wedge\sigma_n} 
\end{equation*}
for some continuous local martingale $M$ starting in $0$. Assuming in addition that $\varphi$ is non-negative implies by Lenglart's inequality \cite[Corollaire II]{Lenglart}
\begin{equation*}
\E\left[\sup_{t\leq \sigma_n} \sqrt{\intT \varphi(u(t))}\right] \leq C \E\left[\sqrt{\intT \varphi(u_0)}\right].
\end{equation*}
Let $f\in C^2(\R, [0,\infty))$ denote a convex function with bounded second derivative such that $f|_{(-\infty, 0]}=0$, $f(z)>0$ for all $z>0$.
Choosing $\varphi(u) := f(u-\|u_0\|_\infty)$ and $\varphi(u) := f(-\|u_0\|_\infty-u)$ implies \eqref{eq:Lemma55-proof-1}.

\bigskip
Let $\varphi(v):=|v|^2$, $j\leq k$, $\beta\in\N^d$, $|\beta|=j$, and apply the It\^o formula \cite[Proposition A.1]{Ito-DebusscheHofmanovaVovelle}:
\begin{equation}\label{eq:lemma54-eq1}
\begin{aligned}
\intT \varphi((\partial^j_{\beta} u)(t\wedge \sigma_n))
& = \intT \varphi(\partial^j_{\beta} u_0) - \frac{1}{2} \inttn \intT \sum_{\iota\geq 1} \varphi''(\partial^j_{\beta} u) (\grad\partial^j_\beta u)  \cdot \partial^j_\beta( \g'_{\iota} \Div(\g_{\iota}))\dd s \\
& \quad - \inttn \intT \varphi''(\partial^j_{\beta} u) |\grad\partial^j_{\beta}u|^2 \dd s  +  M_{t\wedge\sigma_n}\\
&  \quad
 + \frac{1}{2} \sum_{\iota\geq 1} \inttn  \intT \varphi''(\partial^j_{\beta} u) (\partial^j_{\beta} \Div(\g_{\iota}))^2 \dd s
\end{aligned}
\end{equation}
where $M$  defined by $M_t := - \sum_{\iota\geq 1} \inttn \intT  \varphi'(\partial^j_{\beta} u) (\partial^j_{\beta} \Div(\g_{\iota})) \dd B^{\iota}_s$ is a martingale with $M_0=0$ (which can be verified using e.g. the definition of $\sigma_n$). 
We observe that
$$\pm \frac{1}{2} \inttn \intT \sum_{\iota\geq 1} \varphi''(\partial^j_{\beta} u) (\grad\partial^j_\beta u)   \cdot \g'_{\iota} \partial^j_\beta(\Div(\g_{\iota}))\dd s$$
is contained in the second and the last term on the right-hand side of \eqref{eq:lemma54-eq1} and hence cancels, resulting in
\begin{equation*}
\begin{aligned}
\intT \varphi((\partial^j_{\beta} u)(t\wedge\sigma_n))
 &=\intT \varphi(\partial^j_{\beta} u_0) -  \inttn \intT \varphi''(\partial^j_{\beta} u) |\grad \partial^j_{\beta} u|^2 \dd t
+ M_{t\wedge\sigma_n} \\
& \quad -\frac{1}{2} \inttn \intT \sum_{\iota \geq 1} \varphi''(\partial^j_{\beta} u) ( \grad \partial^j_\beta u) \cdot
\bigg\{\sum_{\substack{\beta_1 + \beta_2 = \beta,\\ |\beta_1| \neq 0}}
c_{\beta_1, \beta_2} \partial^{|\beta_1|}_{\beta_1} (\g'_{\iota}) \partial^{|\beta_2|}_{\beta_2} \Div(\g_{\iota})\bigg\} \dd t\\
& \quad + \frac{1}{2} \sum_{\iota \geq 1} \inttn \intT \varphi''(\partial^j_{\beta} u) \bigg\{\sum_{\substack{\beta_1 + \beta_2 = \beta,\\ |\beta_1| <j}}c_{\beta_1, \beta_2}
\partial^{|\beta_1|}_{\beta_1} (\grad u)  \cdot  \partial^{|\beta_2|}_{\beta_2} (\g'_{\iota}) \bigg\}(\partial^j_{\beta} \Div(\g_{\iota})) \dd t 
\end{aligned}
\end{equation*}
where $c_{\beta_1, \beta_2}$ denote suitable constants. We observe that all terms that contain a product of two partial derivatives of $u$ of order $j+1$, apart from the Laplacian terms, have cancelled. We apply Hölder's inequality, \eqref{eq:Pv-estimate} of \cref{rmk:simpleInequality} and sum over $|\beta|\leq k$, yielding
\begin{equation*}
\begin{aligned}
\|u(t\wedge\sigma_n)\|_{H^k}^2
 &\leq \|u_0\|_{H^k}^2 -  2 \inttn \|\grad u\|_{H^{k}}^2 \dd t
+ M_{t\wedge\sigma_n} \\
& \quad +\delta \inttn \|u\|_{H^{k+1}}^2 \dd t + C(\delta, g) \inttn (1+ \|u\|_{H^k}^2)(1+\|u\|_{X_\theta}^{2/\theta}) \dd t,
\end{aligned}
\end{equation*}
and hence the first claim of the lemma. Applying instead \eqref{eq:Pv-estimate2} implies the second claim.

\end{proof}

For the convenience of the reader we cite here a special case of the stochastic Gronwall inequality provided in \cite[Definition 2.1, Definition 2.2, Corollary 5.4 b)]{ConvexGronwall}:
\begin{lemma}\label{lemma:StochGron}
Assume that $(\Omega, \PP, \F, (\F_t)_{t\geq 0 })$ is an underlying filtered probability space satisfying the usual conditions. Let
\begin{enumerate}
     \item $(X_t)_{t\geq 0}$ be an adapted continuous non-negative  process,
     \item $(A_t)_{t\geq 0}$ be an adapted continuous non-negative non-decreasing process,
     \item $(H_t)_{t\geq 0}$ be an adapted continuous non-negative non-decreasing process,
     \item $(M_t)_{t\geq 0}$ be a c\`adl\`ag local martingale with $M_0=0$.
    \end{enumerate}
    Assume that for all $t\geq 0$ we have
    \begin{equation*}
     X_t \leq \int_{(0,t]} X_s \dd A_s + M_t + H_t  \qquad \PP-a.s.
    \end{equation*}
    Then, we have for all $T>0$, $z>0$, $w>0$ and $R>0$
    \begin{equation*}
    \PP\left[\sup_{t\in[0,T]} X_t \geq z\right] \leq \frac{\e^R}{z}\E[H_T\wedge w] + \PP[H_T \geq w] + \PP[A_T>R]
    \end{equation*}    
    and, in particular, $\PP[A_T <\infty, H_T< \infty] = 1$ implies $\PP[\sup_{t\leq T} X_t <\infty] = 1$.
\end{lemma}

The subsequent lemma shows local well-posedness and the blow-up criterion explained in the introduction of this section.
For related results in the stochastic maximal-regularity framework see for example local well-posedness  with $W^{2-2(1+\kappa)/p,p}$ spacial regularity \cite[Theorem 6.4]{AgrestiVeraar2022-Part1}, instaneuous regularization \cite[Theorem 6.3]{AgrestiVeraar2022-Part2}, and a blow-up criterion \cite[Theorem 4.9]{AgrestiVeraar2022-Part2}.

\begin{lemma}\label{thm:local-in-spaceEquationWP}
Assume \cref{assumption} and $u_0\in  L^2(\Omega, \F_0, \PP; H^k(\T^d))$. Then the following assertions hold:
\begin{enumerate}
\item There exists a unique maximal solution $(u,\sigma)$ of \eqref{eq:SPDElocal} in the sense of \cref{def:solHk}.

\item Choose $\theta$ such that $0<\theta<1/k$ and let $X_\theta$ be a fractional Sobolev space $W^{s,p}(\T^d)$, $p\in[1,\infty), s\in(0,k]$, with Sobolev index $\theta:=s-d/p$.
The unique maximal solution $(u,\sigma)$ satisfies for any $T>0$
 \begin{equation*}
 \PP\left[\int_0^{\sigma\wedge T} \|u(t)\|_{X_\theta}^{2/\theta} \dd t  < \infty, \sigma \leq T\right] = 0.
\end{equation*}

\item The unique maximal solution $(u,\sigma)$ satisfies for any $T>0$
\begin{equation*}
 \PP\left[\int_0^{\sigma\wedge T} \|u(t)\|_{W^{1,\infty}(\T^d)}^2 \dd t  < \infty, \sigma \leq T\right] = 0.
\end{equation*}
\end{enumerate}
\end{lemma}

\begin{remark} Let $\alpha\in(0,1)$. Noting $\|u\|_{W^{s,p}(\T^d)}\leq C \|u(t)\|_{C^\alpha(\T^d)}$ for $p\in[1,\infty)$, $0<s<\alpha<1$ implies that $L_t^\infty C_x^\alpha(\T^d)$ regularity prevents the blow-up of the $H^k(\T^d)$ norm.   
\end{remark}

\begin{remark}[Global existence for large viscosity]\label{rmk:globalMaxiSol}
Under suitable regularity and decay assumptions on $(g_\iota)_{\iota\geq 1}$, there exists a constant $C_g\geq 0$, depending on the coefficients, such that, for every $\nu>C_g$ and every non-negative initial datum $u_0\in L^2(\Omega,\F_0,\PP;H^k(\T^d))$,
the unique maximal solution $(u,\sigma)$ to
\begin{equation*}
\dd u
= -\sum_{\iota\geq 1}\Div\bigl(ug_\iota(u,x)\bigr)\circ\dd W_t^\iota
  +\nu\Delta u\,\dd t,
\qquad u(0)=u_0,
\end{equation*}
is global, that is, $\sigma=\infty$ almost surely. One way to establish this is to derive, for some $p>d$ and every $T>0$, an a priori estimate ensuring that
\[
\sup_{0\leq t<\sigma\wedge T}\|u(t)\|_{W^{1,p}}^p<\infty
\qquad\PP\text{-almost surely}.
\]
The blow-up criterion in \cref{thm:local-in-spaceEquationWP}~2) then applies by setting $X_\theta = W^{s,p}(\T^d)$ for some suitable $s\leq 1$.
\end{remark}

\begin{proof}[Proof of \cref{thm:local-in-spaceEquationWP}]
\textit{1)} By standard arguments we deduce the claim from the well-posedness of the tamed SPDE: For any $R\in\N$ denote 
by $(u^R(t))_{t\geq 0}$ the unique variational solution of the tamed SPDE \eqref{eq:approx-cut-local}. Define
\begin{equation}
         \sigma^R := \inf\{t\geq 0 \mid \|u^R(t)\|_{H^k(\T^d)} \geq R\}.
\end{equation}
By uniqueness of local-in-time solutions ensured by \cref{thm:local-in-spaceEquationCutoffWP}, we have $u^R = u^{R'}$ 
$\PP$-almost surely on $[0,\sigma^{R}]$ for $R\leq R'$. It can be easily verified that $\sigma := \lim_{R\to\infty} \sigma^R$ 
and $u(t\wedge\sigma^R) := u^R(t\wedge\sigma^R)$ define the unique maximal solution.

\textit{2)} We first show  $\PP\left[\int_0^{\sigma\wedge T} \|u(t)\|_{X_\theta}^{2/\theta} \dd t  < \infty, \sigma \leq T\right] = 0$. Define 
\begin{equation*}
\sigma_n := \inf\{t\geq 0 \mid \|u(t)\|_{H^k(\T^d)} \geq n\},  \quad \upsilon_N := \inf\left\{t\in[0,\sigma) \, \bigg| \, \int_0^t  \|u(s)\|_{X_{\theta}}^{2/\theta} \dd s  = N\right\}\wedge\sigma.
\end{equation*}
It suffices to show for any $N\in\N$ and $T>0$
\begin{equation}\label{eq:lemma25-show}
 \PP\left[\sup_{t\in[0,\sigma \wedge T \wedge \upsilon_N)} \|u(t)\|_{H^k(\T^d)} = \infty\right] = 0
\end{equation}
due to
\begin{equation*}
\begin{aligned}
 \PP\left[\int_0^{\sigma\wedge T} \|u(t)\|_{X_{\theta}}^{2/\theta}\dd t  < \infty, \sigma \leq T\right] & = 
\lim_{N\to\infty} \PP\left[\upsilon_N \geq T\wedge \sigma, \sup_{t\in[0,\sigma \wedge T)} \|u(t)\|_{H^k(\T^d)} = \infty\right]  \\
& \leq 
\lim_{N\to\infty} \PP\left[\sup_{t\in[0,\sigma \wedge T\wedge\upsilon_N)} \|u(t)\|_{H^k(\T^d)} = \infty\right].
\end{aligned}
\end{equation*}
We fix some $n, N\in\N$ and apply \cref{lemma:StochGron} to $X_t := \|u(t\wedge\upsilon_N \wedge\sigma_n)\|_{H^k(\T^d)}^2$, $A_t := C_{\theta,g} \int_0^{t\wedge \upsilon_N}(1+\|u\|_{X_\theta})^{2/\theta} ds$ and $H_t:= \|u_0\|_{H^k}^2 +A_t$, yielding  by \cref{lemma:ItoLocalEq} 1)
\begin{equation*}
\begin{aligned}
&\PP\left[\sup_{t\leq T}\|u(t\wedge\upsilon_N \wedge\sigma_n)\|_{H^{k}(\T^d)}^2 \geq z \right]
\leq \frac{\e^R}{z}\E[H_T\wedge w] + \PP[H_T \geq w] + \PP[A_T>R],
\end{aligned}
\end{equation*}
and hence, by taking the limit $n\to\infty$, noting that $H$ and $A$ do not depend on $n$, we observe
\begin{equation*}
\begin{aligned}
& \PP\left[\sup_{t\in[0,\upsilon_N\wedge T]\cap[0,\sigma)}
\|u(t)\|_{H^{k}(\T^d)}^2\geq z \right]  \leq \frac{\e^R}{z}\E[H_T\wedge w] + \PP[H_T \geq w] + \PP[A_T>R]
\end{aligned}
\end{equation*}
As $A_T$ and $H_T$ are $\PP$-a.s. finite, this implies \eqref{eq:lemma25-show}.

\textit{3)} This assertion follows from an proof analogous to that of  \cref{thm:local-in-spaceEquationWP} 2) by applying \cref{lemma:ItoLocalEq} 2) instead of \cref{lemma:ItoLocalEq} 1).
\end{proof}

\subsection{Non-local SPDE: Global well-posedness}\label{subsec:thm:local-in-spaceEquationWP}

Solutions to the  nonlocal equation \eqref{eq:SPDEnonlocal} of lower spatial regularity are known to exist global in time, see \cite{KurtzXiong1999} and \cite[Theorems 5.3 and 5.4]{CoghiGess}. The following lemma ensures solutions that are smooth enough for our analysis.

\begin{lemma}\label{thm:nonlocal-in-spaceEquationWP}
Assume \cref{assumption}, $u_0\in  L^2(\Omega, \F_0, \PP; H^k(\T^d))$ and $\varepsilon\in(0,1)$. Then the following assertions hold:
\begin{enumerate}
\item There exists a unique maximal solution $(u^\varepsilon, \sigma)$ of \eqref{eq:SPDEnonlocal}.
\item If  $u_0 \geq 0$ then $\sigma=\infty$.
\item If  $u_0 \geq 0$  and $\|u_0\|_{L^1} \in L^\infty(\Omega)$, then for any $T>0$ the SPDE \eqref{eq:SPDEnonlocal} has a  variational solution $(u^\varepsilon(t))_{t\in[0,T]}$.
\end{enumerate}
\end{lemma}

\begin{proof}
\textit{1)}
For this proof consider a fixed choice of $\varepsilon\in(0,1)$, $T>0$ and $u_0$. By standard arguments we deduce the existence of a unique maximal solution from the well-posedness of the tamed SPDE, as done in \cref{thm:local-in-spaceEquationWP} 1).

\textit{2) and 3)} We denote this maximal solution by $(u^\varepsilon, \sigma)$ and assume $u_0\geq 0$. We use the notation $\sigma_n := \inf\{t\geq 0 \mid \|u^\varepsilon(t)\|_{H^k(\T^d)} \geq n\}$. We shorten the notation by setting
\begin{equation*}
u_n(t) := u^\varepsilon(t\wedge\sigma_n) \qquad \forall t\in[0,T], n\in\N,
\end{equation*}
dropping the dependence on $\varepsilon>0$ in the notation.

We first prove $\|u_n(t)\|_{L^1(\T^d)} = \|u_0\|_{L^1(\T^d)}$ a.s. for every $t\in[0,T]$. It suffices to show that $u_0 \geq 0$ implies that  $u_n(t) \geq 0$ for any $t$: Define $f\in C^2(\R,\R)$ by the following properties:
\begin{equation*}
f(0) = 0, \quad f'(0) =0 \quad \text{and} \quad 
f''(z) := \begin{cases} -z(z+2) \quad  & -2\leq z \leq 0 \\ 0 & \text{else}.
\end{cases}
\end{equation*}
In particular, $f(u_0) = 0$, $f\geq 0$, $f(z)>0$ for all $z<0$, $f$ has a bounded second derivative, and satisfies for all $z\in\R$ the inequalities $|f''(z)z^2| \leq 12 f(z)$ and $|\int_0^z f''(y)y dy| \leq 3f(z)$. By $k > d/2 + 1$ we have $\|u_n(t)\|_{W^{1,\infty}(\T^d)} \leq C\|u_n(t)\|_{H^k(\T^d)} \leq C n + C\|u_0\|_{H^k(\T^d)}$ for all $t\in[0,T]$. We apply Ito's formula, use cancellations between the Ito to Strato term and the quadratic variation, integration by parts and the two inequalities for $f$, to obtain:
\begin{equation*}
\begin{aligned}
\intT f(u_n(t)) 
& = \intT f(u_0)  - \inttn\intT f''(u_n) |\grad u_n|^2 \dd s   + M_t 
\\
& \quad - \frac{1}{2} \inttn \intT f''(u_n) \grad u_n \cdot 
\sumi  u_n g_\iota'(u_n\falt) \Div(u_n g_\iota(u_n\falt))\falt \dd s  \\
& \quad + \frac{1}{2} \sumi \inttn \intT f''(u_n) 
u_n g'_\iota(u_n\falt) \cdot \grad u_n\falt \Div(u_n g_\iota(u_n\falt))\dd s
\\
& 
\leq  \intT f(u_0) + M_t 
\\
& \quad + \frac{1}{2} \inttn \intT 3f(u_n) \left\|
\sumi  \Div(g_\iota'(u_n\falt) \Div(u_n g_\iota(u_n\falt))\falt)\right\|_{L^\infty(\T^d)}\dd s  \\
& \quad + \frac{1}{2} \inttn \intT 3 f(u_n)  \left\|\sumi 
 \Div(g'_\iota(u_n\falt) \cdot \grad u_n\falt g_\iota(u_n\falt)) \right\|_{L^\infty(\T^d)} \dd s\\
 & \quad + \frac{1}{2} \inttn \intT 12 f(u_n)  \left\|\sumi 
 g'_\iota(u_n\falt) \cdot \grad u_n\falt \Div(g_\iota(u_n\falt)) \right\|_{L^\infty(\T^d)} \dd s\\
& \leq  \intT f(u_0) + M_t  + \int_0^{t\wedge\sigma_n}\intT f(u_n) \dd A_t,
\end{aligned}
\end{equation*}
where
\begin{equation*}
\begin{aligned}
M_t & := - \sumi \int_0^{t\wedge\sigma_n} \intT f'(u_n) \Div(u_n g_\iota(u_n\falt)))  \dd W^\iota(s), \\
C_g &:= \sumi \|g_\iota\|^2_{C^2([-Cn + \|u_0\|_{H^k(\T^d)}, Cn + \|u_0\|_{H^k(\T^d)}]\times\T^d)},\\
A_t &:= C_\varepsilon (Cn + \|u_0\|_{H^k(\T^d)})^2 C_g t. \\
\end{aligned}
\end{equation*}
Here $C_\varepsilon$ denotes a finite constant that depends on $\varepsilon$ and $\eta$.
We obtain by Gronwall's inequality \\
$\E[\int_{\T^d} f(u_n(t)) \one_{\{\|u_0\|_{H^k(\T^d)}<N\}}] = 0$ for all $N\in\N$ and $t\in[0,T]$, implying $u_n(t) \geq 0$ $\PP$-almost surely for all $t\in[0,T]$.

By Ito's formula, \cref{lemma:BoringEstimate} and $\|u_n(t)\|_{L^1(\T^d)} = \|u_0\|_{L^1(\T^d)}$ we obtain:
\begin{equation}\label{eq:lemma23}
\|u^\varepsilon(t\wedge\sigma_n)\|_H^2  + \int_0^{t\wedge\sigma_n} \|u^\varepsilon(s)\|^2_V\dd s  \leq \|u_0\|_H^2 + C \int_0^t \|u^\varepsilon (s\wedge\sigma_n)\|_H^2 \dd s  +  M_t
\end{equation}
for $C :=1+ C_{d,\varepsilon,1} \left(\left(1+\sumi \|g_\iota\|_{C^{k+2}(M_\varepsilon)}^2\right) (1+\|u_0\|_{L^1(\T^d)})^{k+2} \right)^2$ and \\$M_\varepsilon :=  [-\|u_0\|_L^1 \|\eta_{\varepsilon,per}\|_{L^\infty}, \|u_0\|_L^1 \|\eta_{\varepsilon,per}\|_{L^\infty}]\times\T^d$. We observe that $C$ is $\F_0$-measurable and $\PP$-almost surely finite.

We apply e.g. \cref{lemma:StochGron} which yields for all $w>0$, $n\in\N$, $R>0$
\begin{equation*}
\PP[\sigma_n \leq T] = \PP\left[\sup_{t\in[0,T]} \|u_n(t)\|^2_H  \geq  n^2 \right] \leq \frac{\e^R}{n^2} \E\left[\|u_0\|_H^2\wedge w\right] + \PP[\|u_0\|^2_H\geq w] + \PP[ CT > R]
\end{equation*}
which implies $\PP[\sigma \leq T] = 0$, ensuring hence $\sigma > T$ almost surely.
If $\|u_0\|_{L^1(\T^d)} \in L^\infty(\Omega, \F_0, \PP)$, then we obtain by \eqref{eq:lemma23} and Gronwall's inequality $\hat{u}^\varepsilon \in L^2([0,T]\times\Omega, \dd t\otimes\PP, H^{k+1}(\T^d))$, and hence $(u^\varepsilon(t))_{t\in[0,T]}$ is a variational solution.
\end{proof}

\section{Estimates uniform in $\varepsilon>0$ for the nonlocal SPDEs}
\label{sec:5}
The following two lemmata provide estimates which are uniform in $\varepsilon>0$ for the tamed nonlocal-in-space SPDE. Later these will be applied to prove that $(\uer)_{\varepsilon >0}$ has a convergent subsequence via a compactness argument.

\begin{lemma}\label{lemma:H^k-uniform-in-eps}
Assume \cref{assumption} and $u_0 \in L^2(\Omega, \F_0, \PP; H^{k+1}(\T^d))$. Let $(\uer(t))_{t\in[0,T]}$ denote the variational solution of \eqref{eq:approx-cut-nonlocal} for some choice of $R>0$, $\varepsilon\in(0,1)$ and $T>0$. Then, there exists a constant
$K<\infty$ depending on $R$ and $d$, and further constants $C_{1,R}<\infty$, $C_{2,R}<\infty$ depending on $R$, $C_g := \sumi \|g_\iota\|^2_{C^{k+2}([-K,K]\times\T^d, \R^d)}$,$d$ and $\eta$, but which do not depend on $\varepsilon>0$, such that $\PP$-a.s. for all $t\in[0,T]$
\begin{equation*}
\begin{aligned}
& \|\uer(T)\|_{\dot{H}^{k+1}(\T^d)}^2 + \frac{1}{2} \int_0^T \|\uer(t)\|_{\dot{H}^{k+2}(\T^d)}^2 \dd t  \\ 
& \qquad \leq  \|u_0\|_{\dot{H}^{k+1}(\T^d)}^2  + T C_{1,R}   +  C_{2,R} \int_0^T \|\uer(t)\|_{\dot{H}^{k+1}(\T^d)}^2 \dd t +   M^{(\varepsilon)}_T,
\end{aligned}
\end{equation*}
where $(M^{(\varepsilon)}_t)_{t \in[0,T]}$ denotes a continuous local martingale with $M^{(\varepsilon)}_0=0$. Moreover, we have for all $v>0$

\begin{equation*}
\begin{aligned}
& \PP\left[\sup_{t\in[0,T]} \|\uer(t)\|_{H^{k+1}(\T^d)}^2 + \frac{1}{2} \int_0^T \|\uer(t)\|_{H^{k+2}(\T^d)}^2 \dd t \geq 2v \right] \\
& \qquad \leq \frac{\e^{(C_{2,R}+1)T}}{v}\left(T C_{1,R}+ \E[\|u_0\|_{H^{k+1}(\T^d)}^2] +(R+1)^2\right).
\end{aligned}
\end{equation*}
\end{lemma}

In the proof of \cref{lemma:H^k-uniform-in-eps} we will use the following estimate several times; it bounds products of partial derivatives of order $\leq k+1$.
\begin{remark}\label{rmk:trivalEstimateI}
Let $k$, $d$ and $\alpha$ be as in \cref{assumption}, $\beta_1, \ldots, \beta_n\in\N^d_0$, $l_m := |\beta_m|\leq k+1$, $l_1 + \ldots + l_n \leq k+2$, $u\in H^{k+1}(\T^d)$, $f_1, f_2 \in L^\infty(\T^d)$, $v_m \in\{u, u\ast \eta_\varepsilon\}$ for all $m\in\{1, \ldots, n\}$ and $n\geq 1$.
Then we have
\begin{equation}\label{eq:trivalEstimateI}
\begin{aligned}
&\|\partial^{l_1}_{\beta _1}v_1\partial^{l_2}_{\beta_2}v_2
\ldots (\partial^{l_i}_{\beta_i}v_i \cdots \partial^{l_n}_{\beta_n}v_n f_1)\ast \eta_\varepsilon  f_2 \|_{L^2(\T^d)} \\
& \leq C \|u\|_{H^{k+1}(\T^d)}\|u\|^{n-1}_{H^{k}(\T^d)}\|f_1\|_{L^\infty(\T^d)} \|f_2\|_{L^\infty(\T^d)}  \\
\end{aligned}
\end{equation}
for some constant $C$ that only depends on $n$ and $d$. For $i>n$ we use the notation  $(\partial^{l_i}_{\beta_i}v_i \cdots \partial^{l_n}_{\beta_n}v_n f_1)\ast \eta_\varepsilon:=1$ and $\|f_1\|_{\infty}:=1$.

That \eqref{eq:trivalEstimateI} holds can be seen as follows:  Assume w.l.o.g. that $l_1 \geq l_m$ for all  $1<m\leq n$ (by relabeling if necessary). We first show the assertion under the additional assumptions $n\geq 2$, $3\leq l_1$ and $2 \leq l_m$ for all $m>1$: In this case we have by Hölder's inequality, Young's convolution inequality, the Sobolev inequality and $\|\eta_{\varepsilon,per}\|_{L^1(\T^d)}= \|\eta\|_{L^1(\R^d)} = 1$
\begin{equation*}
\begin{aligned}
&\|\partial^{l_1}_{\beta _1}v_1\partial^{l_2}_{\beta_2}v_2
\ldots (\partial^{l_i}_{\beta_i}v_i \cdots \partial^{l_n}_{\beta_n}v_n f_1)\ast \eta_\varepsilon  f_2\|_{L^2(\T^d)} \\
& \leq \|\eta_\varepsilon\|_{L^1(\R^d)} \|\partial^{l_1}_{\beta _1}v_1\|_{L^{q_1}(\T^d)}
\|\partial^{l_2}_{\beta_2}v_2\|_{L^{q_2}(\T^d)} \cdots \|\partial^{l_n}_{\beta_n}v_n\|_{L^{q_n}(\T^d)} \|f_1\|_{L^\infty(\T^d)} \|f_2\|_{L^\infty(\T^d)} \\
& \leq C \|u\|_{H^{k+1}(\T^d)}\|u\|^{n-1}_{H^{k}(\T^d)}\|f_1\|_{L^\infty(\T^d)} \|f_2\|_{L^\infty(\T^d)} \\
\end{aligned}
\end{equation*}
for some constant $C$ depending on $d$, where we choose $q_m\in[2,\infty)$ such that $q_1^{-1} + \ldots + q_n^{-1} =\frac{1}{2}$ and $\frac{1}{q_m}\geq \frac{1}{\tilde{q}_m} := \frac{1}{2} + \frac{l_m-(k+1)}{d} = \frac{l_m-\alpha -1}{d}\in [0,\frac{1}{2}]$ for $m=1$ and $\frac{1}{q_m}\geq \frac{1}{\tilde{q}_m} := \frac{1}{2} + \frac{l_m-k}{d} = \frac{l_m-\alpha}{d} \in [0,\frac{1}{2}]$ for $m>1$. We can find such $q_m$ using $n\geq 2$, $k= d/2+\alpha$ and $\alpha >1$ due to:
\begin{equation*}
 \sum_{m=1}^n \frac{1}{\tilde{q}_m} =  \sum_{m=1}^n \bigg(\frac{l_m -\alpha}{d}\bigg)  - \frac{1}{d} \leq \frac{k+1- n\alpha}{d} = \frac{1 + d/2 + \alpha - n\alpha}{d} < 1/2. 
\end{equation*}

Now we show that we can remove the additional assumptions: Clearly assertion \eqref{eq:trivalEstimateI} also holds true for $n=1$.
Terms $\partial_{\beta_m}^{l_m} v_m$, $l_m\leq 1$ can be pulled out of the norm using $\|u\|_{W^{1,\infty}(\T^d)} \leq C \|u\|_{H^{k}(\T^d)}$, allowing us to remove the extra assumption  $l_m\geq 2$ for all $m>1$. To prove the assertion for $l_1 \leq 2$, it suffices to show the claim for the case $l_m=2$ for all $m$. In this case, choose $q_m = 2n$ for all $m$, which yields
\begin{equation*}
\begin{aligned}
\|\partial^{2}_{\beta _1}v_1\partial^{2}_{\beta_2}v_2
\ldots (\partial^{2}_{\beta_i}v_i \cdots \partial^{2}_{\beta_n}v_n f_1)\ast \eta_\varepsilon  f_2\|_{L^2(\T^d)} 
 \leq \|u\|_{\dot{W}^{2,2n}(\T^d)}^n  \|f_1\|_{L^\infty(\T^d)} \|f_2\|_{L^\infty(\T^d)} \\
\end{aligned}
\end{equation*}
In the case of $\alpha=2$ we have $\|u\|_{W^{2,q}(\T^d)} \leq C\|u\|_{H^{k}(\T^d)}$ for all $2\leq q<\infty$ which implies the claim. For $\alpha = 1.5$ we have $\|u\|_{W^{2,q}(\T^d)} \leq C \|u\|_{H^{k}(\T^d)}$ for all $q\in[2,2d]$ for $\alpha = 1.5$. Thus, for $\alpha = 1.5$ and $d\geq 3$ we have $2n \leq k+2 = d/2 + 3.5 \leq 2d$ which yields the assertion.
For $\alpha = 1.5$ and $d=1$ we have $k=2$, $n\leq 2$. For $n=1$ the assertion is clear; for $n=2$ we use
\begin{equation*}
\begin{aligned}
\|\partial^{2}_{\beta _1}v_1\partial^{2}_{\beta_2}v_2
\ldots (\partial^{2}_{\beta_i}v_i \cdots \partial^{2}_{\beta_n}v_n f_1)\ast \eta_\varepsilon  f_2\|_{L^2(\T^d)} 
 &\leq \|\partial^{2}_{\beta _1}v_1\|_{L^\infty(\T^d)} \|\partial^{2}_{\beta_2}v_2\|_{L^2} \|f_1\|_{L^\infty(\T^d)} \|f_2\|_{L^\infty(\T^d)} \\
 &\leq C \|u\|_{H^{k+1}} \|u\|_{H^{k}} \|f_1\|_{L^\infty(\T^d)} \|f_2\|_{L^\infty(\T^d)} .
\end{aligned}
\end{equation*}
\end{remark}

\begin{proof}[Proof of \cref{lemma:H^k-uniform-in-eps}]
We fix some $\varepsilon>0$, $R<\infty$ and $T>0$. We assume w.l.o.g. that $\|g_\iota\|_{C^{k+2}} \neq 0$ for all $\iota\geq 1$ and that $\cup_{\iota \geq 1}\textrm{supp}(g_\iota)\subset [-K,K]\times \T^d$ for some $K<\infty$ which depends on $R$ and $d$ but not on $\varepsilon$. For the proof shorten the notation by $u:=\uer$, $g_\iota := g_\iota(\uerp\falt)$ and $g'_\iota := g'_\iota(\uerp\falt)$ except in the notation $\|g_\iota\|_{C^{k+2}} := \|g_\iota\|_{C^{k+2}(\R\times\T^d, \R^d)}$.  We fix some $\alpha\in\N_{0}^d$ with $|\alpha|=k+1$ and apply  the Ito formula \cite[Proposition A.1]{Ito-DebusscheHofmanovaVovelle} to $\partial^{k+1}_\alpha u$, for which we have sufficient regularity by \cref{thm:nonlocal-in-spaceEquationCutoffWP} and \cref{lemma:BoringEstimate}. This gives for all $t\in[0,T]$
\begin{equation}\label{eq:uniform-in-eps-Hk+1-ito}
\begin{aligned}
\frac{1}{2} \|\partial^{k+1}_\alpha u(t)\|^2_{L^2(\T^d)} & = \frac{1}{2} \|\partial^{k+1}_\alpha u_0\|^2_{L^2(\T^d)}  \\
& \quad - \frac{1}{2} \int_0^t   \intT (\grad \partial^{k+1}_\alpha \up) \cdot \sum_{\iota \geq 1}\gRp^2 \partial^{k+1}_\alpha (g_\iota \Div(\up g_\iota))\dd s \\
& \quad - \frac{1}{2} \int_0^t \intT (\grad \partial^{k+1}_\alpha \up)\cdot  \sum_{\iota \geq 1} \gRp^2   \partial^{k+1}_\alpha
(\up g'_\iota \Div(\up g_\iota)\ast\eta_\varepsilon)\dd s \\
& \quad
-   \int_0^t\intT |\grad \partial^{k+1}_{\alpha} \up|^2 \dd s +  M^{(\varepsilon)}_t \\
& \quad 
+ \frac{1}{2}\sum_{\iota \geq 1}    \int_0^t \gRp^2 \intT |\partial^{k+1}_\alpha \Div(\up g_\iota)|^2 \dd s
\end{aligned}
\end{equation}
where 
\begin{equation}\label{eq:uniform-in-eps-Hk+1-mart}
M^{(\varepsilon)}_t := - \int_0^t  \sum_{\iota \geq 1} \gRp \intT \partial^{k+1}_\alpha \up \partial^{k+1}_\alpha \Div(\up g_\iota)) dW^\iota(s)
\end{equation}
which is a local martingale with $M^{(\varepsilon)}_0=0$. There are no terms that contain derivatives of order $>k+2$. We collect the terms which contain a product of two partial derivatives of order $k+2$, and interchange sums and integrals by Fubini:
\begin{equation*}
\begin{aligned}
& \frac{1}{2} \|\partial^{k+1}_\alpha \up (t)\|^2_{L^2(\T^d)} \\
& \quad = \frac{1}{2} \|\partial^{k+1}_\alpha u_0\|^2_{L^2(\T^d)}   + M^{(\varepsilon)}_t   + T^{\text{lower-order}}_t&&\\
& \qquad - \frac{1}{2} \int_0^t \sumi  \gRp^2   \intT (\grad \partial^{k+1}_\alpha \up) \cdot  g_{\iota} \partial^{k+1}_\alpha \Div(\up g_{\iota}) \dd s
&&
\text{Strato to Ito (I)}\\
& \qquad-\frac{1}{2} \int_0^t\sumi \gRp^2   \intT (\grad \partial^{k+1}_\alpha \up) \cdot
(\up g'_{\iota} \{\partial^{k+1}_\alpha \grad \up \cdot g_{\iota} \}\ast\eta_\varepsilon) \dd s  && \text{Strato to Ito (II)}\\
& \qquad-\frac{1}{2} \int_0^t\sumi \gRp^2   \intT (\grad \partial^{k+1}_\alpha \up)
\cdot (\up g'_{\iota} \{(\partial^{k+1}_\alpha \grad \up\falt) \cdot \up g'_{\iota}\}\ast\eta_\varepsilon) \dd s
&& \text{Strato to Ito (III)} \\
&\qquad + \frac{1}{2} \int_0^t \sumi \gRp^2 \intT(\partial^{k+1}_\alpha \grad \up)\cdot  g_{\iota} \partial^{k+1}_\alpha \Div(\up g_{\iota})  \dd s && \text{Ito quadr. variation (I)}\\
&\qquad + \frac{1}{2}\int_0^t \sumi \gRp^2 \intT (\partial^{k+1}_\alpha \grad \up \falt)   \cdot \up g'_{\iota} (\partial^{k+1}_\alpha\grad \up \cdot g_{\iota})  \dd s
&& \text{Ito quadr. variation (II) }\\
& \qquad+  \frac{1}{2}\int_0^t \sumi \gRp^2 \intT (\partial^{k+1}_\alpha \grad \up \falt)    \cdot  \up g'_{\iota} (\partial^{k+1}_\alpha \grad  \up\ast \eta_\varepsilon\cdot  \up g'_{\iota}) \dd s
&& \text{Ito quadr. variation (III)}\\
& \qquad -  \int_0^t \intT  |\grad \partial^{k+1}_{\alpha} \up |^2  \dd s   && \text{Laplacian}\\
\end{aligned}
\end{equation*}
The terms 'Strato to Ito (I)' and  'Ito quadratic variation (I)' cancel.

\smallskip

\noindent \emph{Upper bound for $T^{\text{lower-order}}$}: By construction $T^{\text{lower-order}}$ is a sum of terms that contain at most one derivative of order $k+2$. First, we provide an estimate for the terms that contain exactly one derivative of order $k+2$. These come from the second, third and last term on the right-hand side of \eqref{eq:uniform-in-eps-Hk+1-ito} and are of the form:
\begin{equation*}
\begin{aligned}
\int_0^t  \sumi   \gRp^2 \intT (\partial_i\partial^{k+1}_\alpha v_0) \partial^{l_1}_{\beta _1}v_1\partial^{l_2}_{\beta_2}v_2
\ldots (\partial^{l_i}_{\beta_i}v_i \cdots \partial^{l_n}_{\beta_n}v_n \tilde{f}_{1,\iota})\ast \eta_\varepsilon \tilde{f}_{2,\iota} \dd t
\end{aligned}
\end{equation*}
where $l_1 + \cdots + l_n \leq k+2$,  $v_i \in\{\bar{u}, \bar{u}\ast \eta_\varepsilon\}$,  and $\tilde{f}_{1,\iota}$ and $\tilde{f}_{2,\iota}$ satisfy $\|\tilde{f}_{1,\iota} \|_{L^\infty(\T^d)}\|\tilde{f}_{2,\iota}\|_{L^\infty(\T^d)} \leq \|g_\iota\|_{C^{k+2}}^2(1+R)^2$. Here we use the notation $(\partial^{l_i}_{\beta_i}v_i \cdots \partial^{l_n}_{\beta_n}v_n \tilde{f}_{1,\iota})\ast \eta_\varepsilon:= 1$ and $\|\tilde{f}_{1,\iota}\|_{L^\infty(\T^d)} := 1$ if $i>n$. An application of \cref{rmk:trivalEstimateI} ensures the upper bound
\begin{equation*}
\begin{aligned}
& \sumi \tilde{\delta} \|g_\iota\|^2_{C^{k+2}} \int_0^t  \|\partial_i\partial^{k+1}_\alpha u\|_{L^2(\T^d)}^2 \dd t \\
& +  \sumi \frac{1}{4\tilde\delta \|g_\iota\|^2_{C^{k+2}}} \int_0^t  \gRp^4  \|\partial^{l_1}_{\beta _1}v_1\partial^{l_2}_{\beta_2}v_2
\ldots (\partial^{l_i}_{\beta_i}v_i \cdots \partial^{l_n}_{\beta_n}v_n \tilde{f}_{1,\iota})\ast \eta_\varepsilon \tilde{f}_{2,\iota}\|_{L^2(\T^d)}^2 \dd s \\
& \leq C_g \tilde{\delta} \int_0^t  \|\partial_i\partial^{k+1}_\alpha u\|_{L^2(\T^d)}^2 \dd t +\frac{C_g}{4\tilde\delta} C_{-1,R} \int_0^t \gRp \|u\|_{H^{k+1}(\T^d)}^2 \dd s
\end{aligned}
\end{equation*}
for all $\tilde{\delta}>0$, where $C_{-1,R}$ depends on $d$ and $R$ but not on $\varepsilon$. 

It remains to estimate the terms that only contain derivatives of order $\leq k+1$; they come from the last term in \eqref{eq:uniform-in-eps-Hk+1-ito}. These can be easily brought by Young's inequality into the form
\begin{equation*}
\begin{aligned}
\int_0^t  \sum_{|\beta_1|+\ldots+|\beta_n|\leq k+2} C_{\beta_1, \ldots,\beta_n}   \sum_{\iota \geq 1} \gRp^2 \|\partial^{l_1}_{\beta _1}v_1\partial^{l_2}_{\beta_2}v_2 \cdots \partial^{l_n}_{\beta_n}v_n \tilde{f}_{\iota}\|_{L^2}^2 \dd t
\end{aligned}
\end{equation*}
where $\|\tilde{f}_{\iota}\|\leq \|g_\iota\|_{C^{k+2}} (1+R)$. These terms can again be estimated by \cref{rmk:trivalEstimateI}. All in all, we obtain for all $t\geq 0$
\begin{equation*}
T^{\text{lower-order}}_t \leq
\delta C_g \int_0^t \|\bar{u}(s)\|_{\dot{H}^{k+2}(\T^d)}^{2} \dd s  + (1+\frac{1}{4\delta}) C_g C_{0,R} \int_0^t \gRp \|\bar{u}(s)\|_{H^{k+1}(\T^d)}^2 \dd s
\end{equation*}
for all $\delta>0$ where $C_{0,R}$ is a constant that depends on $d$ and $R$ but not on $\varepsilon$.

\smallskip

\noindent \emph{Upper bounds for the highest-order terms}: Using that $\eta$ is symmetric for the second equality, we have for the expressions inside $\int_0^t \sumi \ldots \dd s$ of 
'Ito quadratic variation (II)+Strato to Ito (II)':
\begin{equation*}
\begin{aligned}
&  |\text{Ito quadratic variation (II)}_{s,\iota}+\text{ Strato to Ito (II)}_{s,\iota}| \\
& \quad = \gR^2 \frac{1}{2} \left|\intT (\partial^{k+1}_\alpha\grad  u\ast\eta_\varepsilon) \cdot  u g'_{\iota} (\partial^{k+1}_\alpha \grad u \cdot g_{\iota})
- \intT (\grad \partial^{k+1}_\alpha u)
\cdot (u g'_{\iota} \{\partial^{k+1}_\alpha \grad u \cdot g_{\iota}\}\ast\eta_\varepsilon)  \right|
\\
& \quad  = \gR^2
\frac{1}{2}\left|\intT (\partial^{k+1}_\alpha \grad u\ast\eta_\varepsilon) \cdot  ug'_{\iota} (\partial^{k+1}_\alpha \grad u \cdot g_{\iota})
-\intT \{(\grad \partial^{k+1}_\alpha u) \cdot
(ug'_{\iota})\}\ast\eta_\varepsilon (\partial^{k+1}_\alpha \grad u \cdot g_{\iota})  \right| \\
&\quad  = \gR^2
\frac{1}{2}\left|\intT (\partial^{k+1}_\alpha  u\ast(\partial_i\eta_\varepsilon)) u(g'_{\iota})^i (\partial^{k+1}_\alpha \grad u \cdot g_{\iota})
-\intT \{(\partial^{k+1}_\alpha u)
(u(g'_{\iota})^i)\}\ast(\partial_i\eta_\varepsilon) (\partial^{k+1}_\alpha \grad u\cdot g_{\iota})  \right| \\
& \qquad + \text{ lower order terms} \\
\end{aligned}
\end{equation*}
where the "lower order terms" (after summation over $\iota$) can be bounded in the same way as $T^{\text{lower-order}}$.

For $\tilde{\eta}\in C(\R^d, \R)$ with compact support, $\tilde{\eta}_\varepsilon(x):= \frac{1}{\varepsilon^{d}} \tilde{\eta}(x/\varepsilon) \,\, \forall x\in\R^d$, $v \in L^1(\T^d, \R)$, $w\in C^1(\T^d,\R)$ we have
\begin{equation*}
|(vw)\ast \tilde{\eta}_{\varepsilon}(x) - v\ast \tilde{\eta}_{\varepsilon}(x)w(x)|\leq C_{\tilde{\eta}} \, \varepsilon (|v|\ast|\tilde{\eta}_\varepsilon|)(x) \|Dw\|_{L^\infty(\T^d)},
\end{equation*}
where $C_{\tilde{\eta}}$ is a constant that depends on the support of $´\tilde{\eta}$.
Moreover, we have $\|D(\up (g'_{\iota})^i)\|_{L^\infty(\T^d)} \leq C_d (\|\up\|_{H^k} + \|\up\|^2_{H^k})\|g_\iota\|_{C^{k+2}}$. Hence, we have
\begin{equation*}
\begin{aligned} 
& |\text{Ito quadratic variation (II)}_{s,\iota} +\text{ Strato to Ito (II)}_{s,\iota}| \\
& \quad \leq C_\eta 
\frac{1}{2}  \gR^2\intT |\partial^{k+1}_\alpha  u|\ast|\partial_i\eta_\varepsilon| \varepsilon \|D(u(g'_{\iota})^i)\|_{L^\infty(\T^d)} |\partial^{k+1}_\alpha \grad u| \|g_{\iota}\|_{L^\infty(\T^d)} + \text{ lower order terms} \\
& \quad \leq \delta \|g_\iota\|^2_{C^{k+2}}\|\up \|_{\dot{H}^{k+2}(\T^d)}^2  +  \gR^4 \frac{\|g_\iota\|^2_{C^{k+2}}}{4\delta} \||\partial^{k+1}_\alpha  u|\ast|\partial_i\eta_\varepsilon|\|^2_{L^2(\T^d)} \varepsilon^2 C_{4,R}  + \text{ lower order terms} \\
& \quad \leq \delta \|g_\iota\|^2_{C^{k+2}}\|\up \|_{\dot{H}^{k+2}(\T^d)}^2  +  \gR^4 \frac{\|g_\iota\|^2_{C^{k+2}}}{4\delta} \|\partial^{k+1}_\alpha  u\|^2_{L^2(\T^d)} \|\partial_i\eta_\varepsilon\|^2_{L^1(\R^d)} \varepsilon^2 C_{4,R}   + \text{ lower order terms} \\
& \quad \leq \delta \|g_\iota\|^2_{C^{k+2}} \|\up \|_{\dot{H}^{k+2}(\T^d)}^2  +   \|g_\iota\|^2_{C^{k+2}} \frac{C_{5,R}}{\delta}   \|\partial^{k+1}_\alpha  u\|^2_{L^2(\T^d)} + \text{ lower order terms} \\
\end{aligned}
\end{equation*}
where $C_{4,R}$ and $C_{5,R}$ are constants that depend on $R$, $\textrm{supp}(\eta)$, $\|\eta\|_{W^{1,1}(\R^d)}$ and $d$. Similarly, we have:
\begin{equation*}
\begin{aligned}
& |\text{Ito quadratic variation (III)}_{s,\iota} +\text{ Strato to Ito (III)}_{s,\iota}| \\
& \quad  \leq \frac{1}{2} \gRp^2 \left|  \intT (\partial^{k+1}_\alpha \grad u\falt) \cdot ug'_{\iota} (\partial^{k+1}_\alpha \grad u\ast \eta_\varepsilon \cdot ug'_{\iota}) - \intT (\grad\partial^{k+1}_\alpha u)
\cdot (ug'_{\iota} \{(\partial^{k+1}_\alpha \grad u\falt)\cdot u g'_{\iota}\}\ast\eta_\varepsilon)
 \right| \\
&\quad =  \frac{1}{2}\gRp^2 \left| \intT (\partial^{k+1}_\alpha  \grad u\falt)\cdot ug'_{\iota} (\partial^{k+1}_\alpha \grad u\ast \eta_\varepsilon \cdot ug'_{\iota}) -
\intT \{(\grad \partial^{k+1}_\alpha u) \cdot (ug'_{\iota}) \}\ast \eta_\varepsilon (\partial^{k+1}_\alpha \grad u\falt)\cdot u g'_{\iota} \right| \\
& \quad \leq  C_\eta \frac{1}{2}\gRp^2 \intT |\grad \partial^{k+1}_\alpha u| \ast \eta_\varepsilon \|D(ug'_{\iota})\|_{L^\infty(\T^d)} |\partial^{k+1}_\alpha \partial_l u\falt| \varepsilon\|u (g'_{\iota})^l\|_{L^\infty(\T^d)} \\
&\quad \leq \delta \|g_\iota\|^2_{C^{k+2}} \|\up \|_{\dot{H}^{k+2}(\T^d)}^2 + \frac{1}{8\delta} \gRp^4  C_{6,R} \|g_\iota\|^2_{C^{k+2}}  \|\partial^{k+1}_\alpha u \|_{L^2(\T^d)}^2 \|\partial_l\eta_\varepsilon\|_{L^1(\R^d)}^2 \varepsilon^2\\
&\quad \leq \delta \|g_\iota\|^2_{C^{k+2}}\|\up \|_{\dot{H}^{k+2}(\T^d)}^2 +  \frac{C_{7,R}}{\delta} \|g_\iota\|^2_{C^{k+2}}\|\partial^{k+1}_\alpha u \|_{L^2(\T^d)}^2
\end{aligned}
\end{equation*}

Summing over $|\alpha|=k+1$, inserting the estimates for the lower order terms and the highest order terms and choosing $\delta$ small enough gives for all $t\in[0,T]$
\begin{equation*}
\begin{aligned}
\|u(t)\|_{\dot{H}^{k+1}(\T^d)}^2 + \frac{1}{2} \int_0^t \|u(s)\|_{\dot{H}^{k+2}(\T^d)}^2  \dd s \leq &  \|u_0\|_{\dot{H}^{k+1}(\T^d)}^2  + C_R \int_0^t \|u(s)\|_{\dot{H}^{k+1}(\T^d)}^2 \dd s  \\
&  +  C_R \int_0^t \gRp \|u(s)\|_{H^{k+1}(\T^d)}^2 \dd s +   M^{(\varepsilon)}_t
\end{aligned}
\end{equation*}
where the constant $C_R$ depends on  $C_g$, $\eta$, $d$ and $R>0$, but not on $\varepsilon$ and $(M^{(\varepsilon)}_t)_{t\in[0,T]}$ is (by slight abuse of notation) the sum over $|\alpha|=k+1$ of the local martingales defined in \eqref{eq:uniform-in-eps-Hk+1-mart}. This implies
the first inequality of  \cref{lemma:H^k-uniform-in-eps}. Moreover, we have by 
$\|u(t)\|_{H^k(\T^d)} \leq \max\{\|u_0\|_{H^k(\T^d)}, R+1\}$ that
\begin{equation*}
\begin{aligned}
\|u(t)\|_{H^{k+1}(\T^d)}^2 + \frac{1}{2} \int_0^t \|u(s)\|_{H^{k+2}(\T^d)}^2  \dd s \leq &  \|u_0\|_{H^{k+1}(\T^d)}^2 + (R+1)^2  + (C_R+1) \int_0^t \|u(s)\|_{H^{k+1}(\T^d)}^2 \dd s  \\
&  +  C_R \int_0^t \gRp \|u(s)\|_{H^{k+1}(\T^d)}^2 \dd s +   M^{(\varepsilon)}_t
\end{aligned}
\end{equation*}
Applying \cref{lemma:StochGron}  implies the second inequality, using that for non-negative processes $\sup_t X^1(t) + \sup_tX^2(t) \leq 2\sup_t(X^1(t) + X^2(t))$ holds true.
\end{proof}


\begin{lemma}\label{lemma:Walpha2}
Assume \cref{assumption}, let $u_0 \in L^2(\Omega, \F_0, \PP; H^{k}(\T^d))$ for $k$ as given by \cref{assumption}, $s>d/2+2$,  $R>0$, $\varepsilon\in(0,1)$ and $T>0$. Let $(\uer(t))_{t\in[0,T]}$ denote a variational solution of \eqref{eq:approx-cut-nonlocal}. For all $\beta\in(0,1/2)$, $p\geq 2$ we have
\begin{equation*}
\E\left[ \|\uer\|_{W^{\beta,p}(0,T;H^{-s}(\T^d))}^2 \right] \leq C_R\left(\E\left[\|u_0\|_{H^k(\T^d)}^2\right]+1\right)
\end{equation*}
where the constant $C_R$ depends on $d$, $g$, $p$, $\beta$, $T$ and $R$, but does not depend on $\varepsilon>0$.
\end{lemma}

\begin{proof} We fix some $\varepsilon>0$, $R<\infty$ and $T>0$. As before, we assume w.l.o.g. $\cup_{\iota \geq 1}\textrm{supp}(g_\iota)\subset [-K,K]\times \T^d$ for some $K<\infty$ which depends on $R$ and $d$ but not on $\varepsilon$. We set \\ $C_g := \sumi \|g_\iota\|^2_{C^{k+2}([-K,K]\times\T^d, \R^d)}$. For the proof shorten the notation by $u:=\uer$ and $g_\iota := g_\iota(\uerp\falt)$. Moreover, we set $Y:=H^{-s}(\T^d)$. The value of the constant $C_R$ changes from line to line and depends on the quantities listed above. We use the inclusion map $i\colon H^k(\T^d) \to H^{-s}(\T^d)$, $i(v)(\varphi) := \langle v, \varphi\rangle_{L^2(\T^d)}$ for all $v\in H^k(\T^d)$, $\varphi \in H^s(\T^d)$ and split $u$ into three parts: Define for all $\varphi \in H^s(\T^d)$, $t\in[0,T]$
\begin{equation*}
\begin{aligned}
\langle u_1(t), \varphi \rangle_{L^2}  &:=  - \frac{1}{2}\int_0^t  \langle \sumi  \gRp^2 g_{\iota} \Div(\up g_{\iota}), \grad \varphi \rangle_{L^2} \dd s  \\
& \qquad - \frac{1}{2} \int_0^t  \langle  \sum_{\iota\geq 1} \gRp^2  \up g'_{\iota} \Div(\up g_{\iota})\ast\eta_\varepsilon, \grad \varphi \rangle_{L^2} \dd s - \int_0^t \langle  \grad u,  \grad \varphi \rangle_{L^2}  \dd s, \\
\langle u_2(t), \varphi \rangle_{L^2} &:= -\sum_{\iota\geq 1} \int_0^t  \langle \gRp \Div(\up g_{\iota}), \varphi \rangle_{L^2} \dd W^\iota_s.
\end{aligned}
\end{equation*}
Using the resolvent $(1 + \sum_{\alpha\in\N_0^d, 1\leq |\alpha| \leq  k}(-1)^{|\alpha|} \partial_{\alpha}^{|\alpha|}\partial_{\alpha}^{|\alpha|})^{-1}$ we immediately obtain that $u = u_0 + u_1 + u_2$, for example in $L^0(\Omega,\F, \PP; C([0,T]; H^{-s}(\T^d))$. By the construction of the tamed SPDEs, we have $\|u(t)\|_{H^{k}(\T^d)} \leq \max\{\|u_0\|_{H^{k}(\T^d)}, R+2\}$ for all $t\in[0,T]$.
Moreover, for a constant $C$ independent of $\varepsilon$, we also obtain $\|v\falt\|_{L^\infty(\T^d)} + \|\grad v\falt\|_{L^\infty(\T^d)} \leq \|v\|_{L^\infty(\T^d)} +\|\grad v\|_{L^\infty(\T^d)} \leq C \|v\|_{H^k(\T^d)}$. Hence,
\begin{align}
&|\frac{1}{2} \langle \sumi \gRp^2 g_{\iota} \Div(\up g_{\iota})+ \sumi  \gRp^2  \up g'_{\iota} \Div(\up g_{\iota})\ast\eta_\varepsilon, \grad \varphi \rangle_{L^2} |  \leq  \|\nabla \varphi\|_{L^\infty(\T^d)}  C_R C_g \label{eq:Walpha-u1-2} \\
&| \langle  \grad u, \grad \varphi \rangle_{L^2}|  \leq C  \|\grad \varphi\|_{L^\infty(\T^d)} \max\{\|u_0\|_{H^{k}(\T^d)}, R+2\} \label{eq:Walpha-u1-3}
\end{align}
and
\begin{equation}\label{eq:Walpha-u2}
\sumi |\langle \gRp \Div(\up g_{\iota}), \varphi_\iota \rangle_{L^2}|^2  \leq \sup_{\iota\geq 1}\|\grad \varphi_\iota \|^2_{L^{\infty}(\T^d)} C_R C_g.
\end{equation}
Hence a simple application of Fubini's theorem, \eqref{eq:Walpha-u1-2} and \eqref{eq:Walpha-u1-3} imply $\PP$-a.s.
\begin{equation*}
\begin{aligned}
\|u_0+u_1\|^p_{W^{\beta,p}(0,T;Y)} & =  \|u_0+u_1\|^p_{L^p(0,T;Y)} + \int_0^T \int_0^T \frac{|u_1(t)-u_1(s)|_Y^p}{|t-s|^{1+\beta p}} \dd t \dd s \\
& \leq C(T, R, C_g)(1+\|u_0\|_{H^{k}(\T^d)})^p +  C(p, \beta, T) (C_R C_g +  \max\{\|u_0\|_{H^{k}(\T^d)}, R+2\})^p.\\
\end{aligned}
\end{equation*}
Moreover, we have by \cite[Lemma 2.1, p.369]{FlandoliGatarek} and \eqref{eq:Walpha-u2}
\begin{equation*}
\begin{aligned}
\E\left[ \|u_2\|_{W^{\beta,p}(0,T;Y)}^{p}\right]
& \leq C(T,p,\beta) \E\left[\int_0^T \left(\sumi \|\gRp \Div(\up g_\iota)\|^2_{Y}\right)^{p/2} \dd s \right]   \leq C(T,p,\beta)T (C_R C_g)^{p/2}.
\end{aligned}
\end{equation*}
\end{proof}

\section{Convergence of the non-local SPDEs to the local SPDE}
\label{sec:6}

We first establish that, for any fixed $R$, the solutions of the tamed SPDE converge. This convergence result is subsequently used to prove \cref{thm:main}.

\begin{lemma}\label{lemma:convergence-of-eq-with-cutoff}
Assume \cref{assumption}, let $u_0\in L^2(\Omega, \F_0, \PP;H^{k+1}(\T^d))$ for $k$ as given by \cref{assumption},  $R>0$ and $T>0$. Let $(\uer(t))_{t\in[0,T]}$ and $(u^{R}(t))_{t\in[0,T]}$ denote the unique variational solutions of \eqref{eq:approx-cut-nonlocal} and \eqref{eq:approx-cut-local} respectively. Then, we have for all $\delta>0$, for all $(\varepsilon_{n})_{n\in\N} \subseteq (0,1)$ satisfying $\lim_{n\to\infty}\varepsilon_n = 0$ that
\begin{equation}
 \lim_{n\to\infty} \PP\left[ \sup_{t\in[0,T]} \|u^{\varepsilon_n, R}(t) - u^R(t)\|_{H^k(\T^d)} >\delta\right] =0.
\end{equation}
\end{lemma}

\begin{proof}  Our proof follows the strategy of the proof of \cite[Theorem 5.29]{FehrmanGess}, adapted to the present nonlocal-to-local convergence setting.
We fix some  $(\varepsilon_{n})_{n\in\N}$ and $R>0$ for this proof and drop the $R$ in notation, i.e. $u_\varepsilon :=\uer$ and $u:=\ur$. 
By replacing $u_0$ by $u_0 \one_{\{\|u_0\|_{H^k}<N\}}$, $N\in\N$ if necessary and using a simple limiting argument, we may assume without loss of generality that  $\|u_0\|_{H^k(\T^d)}\leq C_0$ for some $C_0\in\N$.  We may assume w.l.o.g. that $\cup_{\iota\geq 1} \textrm{supp}(g_\iota)$ is contained in a compact subset of $\R\times \T^d$ since for all $v\in H^k(\T^d)$ with $\gamma_R(\|v\|_{H^k(\T^d)})\neq 0$ we have $\|v\ast\eta_{\varepsilon}\|_{L^\infty(\T^d)}\leq \|v\|_{L^\infty(\T^d)}\leq C_{sob} \|v\|_{H^k(\T^d)} \leq C_{sob}(R+1)$.

We first show that $(u_{\varepsilon_n})_{n\in\N}$ converges in probability to some limit.  Subsequently, we deduce the correct limit i.e. we show that $u_{\varepsilon_n}$ converges to $u$.

\bigskip

\textbf{I Convergence of $(u_{\varepsilon_n})_{n\in\N}$ in probability to some limit:} We view $(u_{\varepsilon_n})_{n\in\N}$  as a family of random variables defined on the fixed underlying probability space $(\Omega, \F, \PP)$ taking values in the Polish space $C([0,T]; H^k(\T^d))$. We apply \cite[Lemma 1.1.]{Gyongy-Krylov} to prove that $(u_{\varepsilon_n})_{n\in\N}$ converges in probability:  By this lemma it suffices to show that for any pair $(m_l, n_l)_{l\in\N}$ such that $m_l, n_l \to \infty$ for $l\to\infty$ there exists a subsequence
$(m_l', n_l')_{l\in\N}$  such that
\begin{equation}
 \textrm{law}(u_{\varepsilon_{m_l'}}, u_{\varepsilon_{n_l'}}) \implies \mu \qquad \text{for } l\to\infty
\end{equation}
where $\mu$ is a probability measure on the space $\bar{X}\times\bar{X}$, for $\bar{X} := C([0,T];H^k(\T^d))$, satisfying $\mu(\{(x,y) \in \bar{X}\times\bar{X}: x=y\} =1$.

\bigskip
\textbf{I.1 Convergence of $\textrm{law}(u_{\varepsilon_{m_l'}}, u_{\varepsilon_{n_l'}})$  to some limit:} Let $(\varepsilon_{n})_{n\in\N}$ and $(m_l, n_l)_{l\in\N}$ be as stated above.
Define $\U:= C([0,T];H^k(\T^d))\times  L^2(0,T, H^{k+1}(\T^d))$. We have the following compact embeddings for $\beta\in(1/3,1/2)$ by \cite[Corollary 5, p. 86]{Simon} 
\begin{equation*}
\begin{aligned}
 W^{\beta, 3}(0,T, H^{-s}(\T^d))\cap L^{\infty}(0,T, H^{k+1}(\T^d)) &\csubset C([0,T]; H^k(\T^d))\\
 W^{\beta, 2}(0,T, H^{-s}(\T^d))\cap L^{2}(0,T, H^{k+2}(\T^d)) &\csubset L^2(0,T, H^{k+1}(\T^d)).
\end{aligned}
 \end{equation*}
Moreover, by \cref{lemma:H^k-uniform-in-eps} and \cref{lemma:Walpha2}, for any $\delta>0$ there exists $Q_\delta>0$ such that 
\begin{equation*}
\begin{aligned}
\inf_{\varepsilon\in(0,1)} \PP &\left[ \|u_\varepsilon\|_{ W^{\beta, 2}(0,T, H^{-s}(\T^d))} \leq Q_\delta, \|u_\varepsilon\|_{ W^{\beta, 3}(0,T, H^{-s}(\T^d))} \leq Q_\delta, \right.\\
& \,\, \left.  \|u_\varepsilon\|_{L^{\infty}(0,T, H^{k+1}(\T^d))} \leq Q_\delta, \|\bar{u}_\varepsilon\|_{L^{2}(0,T, H^{k+2}(\T^d))} \leq Q_\delta \right] \geq 1-\delta.
\end{aligned}
\end{equation*}
This implies that $\forall \delta >0$ there exists a compact set  $K_\delta \subset \U \times \U \times C([0,T];\R)^{\N}$  such that
\begin{equation}
\inf_{l\in\N} \PP\left[ (u_{\varepsilon_{m_l}}, \hat{u}_{\varepsilon_{m_l}}, u_{\varepsilon_{n_l}},\hat{u}_{\varepsilon_{n_l}},(W^\iota)_\iota) \in K_\delta \right] \geq 1- \delta
\end{equation}
where we equip $C([0,T];\R)^{\N}$ with the topology induced by the metric
\begin{equation*}
D((f_\iota)_ {\iota\in\N}, (g_\iota)_ {\iota\in\N}):= \sum_{\iota
=1}^{\infty} 2^{-\iota} \frac{\|f_\iota - g_\iota\|_{C([0,T];\R)}}{1+\|f_\iota - g_\iota\|_{C([0,T];\R)}}.
\end{equation*}
Thus, by Prokhorov's theorem for metric spaces (see e.g. \cite[Chapter 1, Theorem 5.1]{Billingsley}) and the Skorokhod representation theorem for separable metric spaces (see e.g. \cite[Chapter 1, Theorem 6.7]{Billingsley}) there exists a subsequence $(m'_l, n'_l)_l$  and an enlarged probability space $(\tilde{\Omega}, \tilde{\F}, \tilde{\PP})$  with random variables $(v_l^1, \hat{v}_l^1, v_l^2,\hat{v}_l^2, (\tilde{W}^\iota_{l})_\iota )$, $(v^1, \hat{v}^1, v^2, \hat{v}^2,  (\tilde{W}^\iota)_\iota)$ taking values in $\U\times\U\times C([0,T];\R)^{\N}$ satisfying
\begin{equation}\label{eq:lemma31-PasConvergence}
(v_l^1, \hat{v}_l^1, v_l^2,\hat{v}_l^2, (\tilde{W}^\iota_{l})_\iota ) \to (v^1, \hat{v}^1, v^2, \hat{v}^2, (\tilde{W}^\iota)_\iota ) \qquad \tilde{\PP}\text{-a.s. for } l\to\infty
\end{equation} 
and
\begin{equation}\label{eq:lemma31-EqualInLaw}
\textrm{law} (v_l^1, \hat{v}_l^1, v_l^2,\hat{v}_l^2, (\tilde{W}^\iota_{l})_\iota) = \textrm{law}(u_{\varepsilon_{m'_l}},\hat{u}_{\varepsilon_{m'_l}}, u_{\varepsilon_{n'_l}},\hat{u}_{\varepsilon_{n'_l}}, (W^\iota)_\iota) \qquad \text{ for all } l\in\N
\end{equation} 
with respect to the topology  of $\U\times\U\times C([0,T];\R)^{\N}$.
Moreover, w.r.t. the equivalence classes of  $L^2([0,T],H^k(\T^d))$ we have $v_l^i=\hat{v}_l^i$ $\tilde{\PP}$-a.s. for $i=1,2$. Define for all $t\in[0,T]$
\begin{equation*}
 \tilde{\F}^*_t := \sigma(\{v^1(s), v^2(s), \tilde{W}^{\iota}(s) \mid s\leq t, \iota \in\N\})\vee\mathcal{N},
\end{equation*}
where $\mathcal{N}$ denotes all nullsets w.r.t $\tilde{P}$. Let $(\tilde{\F}_t)_{t\geq 0}$ denote the smallest right-continuous complete filtration that contains $(\tilde{\F}^*_t)_{t\geq 0}$.

\bigskip
\textbf{I.2 Determining the limit $\mu$:} 
To conclude convergence of $(u^R_{\varepsilon_n})_n$ in probability, it remains to show $\mu(\{(x,y) \in \bar{X}\times\bar{X}: x=y\}) =1$. It suffices to verify that $v^1$ is a variational solution of \eqref{eq:approx-cut-local} in the sense of \cref{def:approx-cut-sol}. By \cref{thm:local-in-spaceEquationCutoffWP}, which guarantees uniqueness of the variational solutions, we may then deduce $v^1 = v^2$.

\smallskip
\textit{Definition of the stochastic process $M$:} In order to verify that $v^1$ is a variational solution, define $\forall t\in[0,T]$, $\varphi\in L^2(\T^d)$
\begin{equation*}
\begin{aligned}
\langle M(t), \varphi \rangle_{L^2(\T^d)}   & := \langle v^1(t), \varphi \rangle_{L^2(\T^d)}   -   \langle v^1(0), \varphi \rangle_{L^2(\T^d)} - \int_0^{t}\langle \Delta v^1, \varphi \rangle_{L^2(\T^d)}  \dd s  \\
&  \quad - \frac{1}{2}\int_0^{t} \langle \sumi \gamma_R^2( \|v^1\|_{H^{k}}) \Div(\g_\iota'( v^1) \Div(\g_\iota(v^1))), \varphi \rangle_{L^2(\T^d)}  \dd s.
\end{aligned}
\end{equation*}
Using $\|u_0\|_{H^k(\T^d)}\leq C_0$ and equations \eqref{eq:lemma31-EqualInLaw} and \eqref{eq:lemma31-PasConvergence}, we obtain:
\begin{equation}\label{eq:lemma31-HkBdd}
\begin{aligned}
 1 & = \PP[\|u_{\varepsilon_{m'_l}}\|_{H^k(\T^d)} \leq \max(R+2,C_0)] = \PP[\|v^1_l \|_{H^k(\T^d)} \leq  \max(R+2,C_0)] \\
 &  = \PP[\|v^1\|_{H^k(\T^d)} \leq  \max(R+2,C_0)].
\end{aligned}
\end{equation}
This ensures that $M$ is a square integrable $(\tilde{\F}^*_t)_{t\geq 0}$ adapted $L^2(\T^d)$-valued stochastic process.

\smallskip
\textit{$M$ is a martingale w.r.t $(\tilde{\F}^*_t)_{t\geq 0}$:}
Let $f: C([0,T];H^k(\T^d)) \times C([0,T];H^k(\T^d))  \times C([0,T]; \R)^\N \to \R$ be an arbitrary continuous bounded function. Let $u_{\varepsilon_{n_l}}|_{[0,s]}$ denote $[0,T]\ni t \mapsto u_{\varepsilon_{n_l}}(t\wedge s)$ and shorten the notation by
\begin{equation*}
\begin{aligned}
F_l & :=       f(u_{\varepsilon_{m'_l}}|_{[0,s]}, u_{\varepsilon_{n'_l}}|_{[0,s]},W|_{[0,s]}) \\
\tilde{F}_l & :=  f(v_l^1|_{[0,s]},v_l^2|_{[0,s]},  \tilde{W}_l|_{[0,s]}) \\
\tilde{F}   &:= f(v^1|_{[0,s]}, v^2|_{[0,s]}, \tilde{W}|_{[0,s]})
\end{aligned}
\end{equation*}
We show that for all $f$ as above, $s\leq t\leq T$ and $\varphi \in H^2(\T^d)$ we have
\begin{align}
&\tilde\E[ \langle M(t)-M(s), \varphi \rangle_{L^2(\T^d)} \tilde{F}]=0 \label{eq:proof-lemma31-M2eq} 
\end{align}
which implies that $M$ is an $L^2(\T^d)$-valued martingale  w.r.t $(\tilde{\F}^*_t)_{t\geq 0}$. To this end, we observe that due to $\uml$ being a solution of \eqref{eq:approx-cut-nonlocal}, a change from the $H^k(\T^d)$ to the $L^2(\T^d)$ scalar product, and sufficient regularity allowing for $u$ instead of $\bar{u}$ in the coefficients, yields
\begin{equation*}
\begin{aligned}
 0 = & \E\bigg[\bigg\{ \langle \uml(t)- \uml(s), \varphi \rangle -   \int_s^{t}\langle  \uml, \Delta \varphi \rangle \dd s\\
 &\quad  + \frac{1}{2}\int_s^{t} \langle \sumi \gamma^2_R( \|\uml\|_{H^{k}}) g_\iota(\uml\faltl)  
 \Div(\uml g_\iota(\uml\faltl))), \grad \varphi \rangle \dd s \\
 & \quad + \frac{1}{2}\int_s^{t} \langle \sumi \gamma^2_R( \|\uml\|_{H^{k}})  \uml g'_\iota(\uml\faltl) \Div(\uml g_\iota(\uml\faltl))\faltl), \grad \varphi \rangle \dd s \bigg\} F_l \bigg]. 
\end{aligned}
\end{equation*}
By equality in law \eqref{eq:lemma31-EqualInLaw}, we have 
 \begin{equation}\label{eq:lemma31-someformula}
 \begin{aligned}
 0 = & \tilde{\E}\bigg[\bigg\{ \langle v^1_l(t)- v^1_l(s), \varphi \rangle -   \int_s^{t}\langle  v^1_l, \Delta \varphi \rangle \dd s  \\
 & \quad +  \frac{1}{2}\int_s^{t} \langle \sumi \gamma^2_R( \|v^1_l\|_{H^{k}}) g_\iota(v^1_l\faltl) 
 \Div(v^1_l g_\iota(v^1_l\faltl))), \grad \varphi \rangle \dd s \\
 & \quad+ \frac{1}{2}\int_s^{t} \langle \sumi \gamma^2_R( \|v^1_l\|_{H^{k}})  v^1_l g_\iota'(v^1_l\faltl) \Div(v^1_l g_\iota(v^1_l\faltl))\faltl), \grad \varphi \rangle \dd s \bigg\} \tilde{F}_l \bigg]. \\
 \end{aligned}
 \end{equation}
Uniform integrability of all terms in \eqref{eq:lemma31-someformula} is ensured by \eqref{eq:lemma31-HkBdd}.  Due to the compact support of $\eta$ and $\|\eta\|_{L_1} = 1$, we have for some constant $C_\eta$, that depends on the support of $\eta$,
$$\|v^1_l \ast \eta_{\varepsilon_{m'_l}}(t) - v^1(t)\|_{H^1(\T^d)} \leq  \|v^1_l(t)- v^1(t)\|_{H^1(\T^d)} + {\varepsilon_{m'_l}} C_{\eta} \|v^1(t)\|_{H^2(\T^d)}.$$
 This ensures that $v^1_l \ast \eta_{\varepsilon_{m'_l}}$ converges to $v^1$ in $C([0,T];H^1(\T^d))$, due to $k\geq 2$ for every $d\in\N_{>0}$.
This, together with the convergence of $(v_l^1)_l$ in $C([0,T];H^k(\T^d))$, allows us to take the limit $l\to\infty$, yielding
 \begin{equation*}
 \begin{aligned}
 0 = & \tilde{\E}\bigg[\bigg\{ \langle v^1(t)- v^1(s), \varphi \rangle -   \int_s^{t}\langle  v^1, \Delta \varphi \rangle \dd s  \\
 & \quad +  \frac{1}{2}\int_s^{t} \langle \sumi \gamma^2_R( \|v^1\|_{H^{k}}) g_\iota(v^1) 
 \Div(v^1 g_\iota(v^1)), \grad \varphi \rangle \dd s \\
 & \quad+ \frac{1}{2}\int_s^{t} \langle \sumi \gamma^2_R( \|v^1\|_{H^{k}})  v^1 g_\iota'(v^1) \Div(v^1 g_\iota(v^1))), \grad \varphi \rangle \dd s \bigg\} \tilde{F} \bigg], \\
 \end{aligned}
 \end{equation*}
which is \eqref{eq:proof-lemma31-M2eq} by definition of $M$ and $\g$.

\smallskip
\textit{Explicit formula for $M$:}
We have
\begin{align}
&\tilde\E[ \int_0^T \sumi \gamma_R^2( \|\bar v^1\|_{H^{k}(\T^d)}) \|\Div(\g_\iota(\bar{v}^1))\|^2_{L^2(\T^d)} \dd s ] < \infty.  \label{eq:proof-lemma31-M1eq} 
\end{align}
Moreover, using the same arguments as in   \cite[p. 41]{FehrmanGess}, one can verify that for every function $f$ as above and all $s\leq t\leq T$ 
\begin{equation}\label{eq:proof-lemma31-BM1eq}
\tilde\E[(\tilde{W}^\iota_t - \tilde{W}^\iota_s) \tilde{F}] = 0 \quad \text{ and } \quad \tilde\E[(\tilde{W}^\iota_t\tilde{W}^\kappa_t - \tilde{W}^\iota_s\tilde{W}^\kappa_s - \delta_{\iota \kappa}(t-s)) \tilde{F}] = 0.
\end{equation}
Here, $\delta_{\iota \kappa}$ denotes the Kronecker delta. Hence, $(\tilde{W}^\iota)_{\iota\geq 1}$ is a family of independent Brownian motions w.r.t. the filtration $(\tilde{\F}^*_t)_{t\in[0,T]}$ by L\'evy's characterization of Brownian motion \cite[Chapter 4, (3.6) Theorem, (1.23) Proposition, (1.2) Proposition]{RevuzYor}.  Using that  $(\tilde{W}^\iota_t)_{t\geq 0}$ is right-continuous for each $\iota$ implies that $(\tilde{W}^\iota)_{\iota\geq 1}$ is also a family of independent Brownian motions w.r.t.  $(\tilde{\F}_t)_{t\in[0,T]}$.

By an calculation analogous to that for \eqref{eq:proof-lemma31-M2eq}, one can verify that for all $\varphi\in H^{2}(\T^d)$
\begin{align}
&\tilde\E\bigg[ \bigg\{|\langle M(t), \varphi \rangle_{L^2(\T^d)}|^2 -  |\langle M(s), \varphi \rangle_{L^2(\T^d)}|^2 - \int_s^t \sum_{\iota \geq 1} \gamma_R^2(\|v_1 \|_{H^{k}}) |\langle \Div(\g_\iota(v^1)), \varphi 
\rangle_{L^2(\T^d)}|^2 \dd r \bigg\}\tilde{F}\bigg]=0 \label{eq:proof-lemma31-M3eq} \\
&\tilde\E\bigg[ \bigg\{ \tilde{W}^\iota(t) \langle M(t), \varphi \rangle_{L^2(\T^d)} -  \tilde{W}^\iota(s)\langle M(s), \varphi \rangle_{L^2(\T^d)}+ \int_s^t \gamma_R(\|v_1 \|_{H^{k}})  \langle \Div(\g_\iota(v^1)), \varphi \rangle_{L^2(\T^d)}\dd r \bigg\}\tilde{F}\bigg]=0, \label{eq:proof-lemma31-M4eq}
\end{align}
which yields the quadratic variations $[\langle M,\varphi\rangle_{L^2(\T^d)},  \langle M,\varphi\rangle_{L^2(\T^d)}]_t$ and $[ \tilde{W}^\iota, \langle M,\varphi\rangle_{L^2(\T^d)}]_t $. We can now compute the quadratic variation of $(\langle M + \sum_{\iota \geq 1} \gamma_R(\|\bar{v}^1 \|_{H^{k}})  \Div(\g_\iota (\bar{v}^1)) \dd \tilde{W}^\iota_t,\varphi\rangle_{L^2(\T^d)})_{t\in[0,T]}$, implying
\begin{equation}\label{eq:Proof-lemma31-martingale}
M(t) = -\sum_{\iota \geq 1} \int_0^t \gamma_R(\|\bar{v}^1 \|_{H^{k}})  \Div(\g_\iota (\bar{v}^1)) \dd \tilde{W}^\iota_t  \qquad \tilde{\PP}\text{-a.s.}
\end{equation}
seen as an equality in $L^2(\T^d)$.

\smallskip
\textit{$v^1$ is a variational solution of \eqref{eq:approx-cut-local}:} Now we can conclude that $v^1$ is indeed a variational solution: As $(v_1, \hat{v}_1)$ is a $\U$-valued random variable and by our definition of  $(\tilde{\F}_t)_t$, we have that $(v_1(t))_{t\in[0,T]}$ is a continuous $H^k(\T^d)$-valued  $(\tilde{\F}_t)_t$-adapted process. By Fatou, convergence in $L^2([0,T],H^{k+1}(\T^d))$, the deterministic Gronwall inequality and \cref{lemma:H^k-uniform-in-eps}, we have
\begin{equation*}
 \E\left[\int_0^T \|\hat{v}^1(t)\|_{H^{k+1}}^2 \dd t\right] \leq \liminf_{l\to\infty} \E\left[\int_0^T \|\hat{v}^1_l(t)\|_{H^{k+1}}^2 \dd t\right]\leq C_{R,T,u_0}<\infty.
\end{equation*}
Hence, the $\dd t \otimes \dd \tilde{P}$ equivalence class $\hat{v}_1$ satisfies $\hat{v}_1\in L^2([0,T]\times \Omega, \dd t \otimes \dd \tilde{\PP}, H^{k+1}(\T^d))$.
By \eqref{eq:Proof-lemma31-martingale} we have that for all $\varphi\in C^\infty(\T^d)$ we have
\begin{equation*}\label{eq:Proof-lemma31-martingale-0}
 \begin{aligned}
\langle v^1(t), \varphi \rangle_{L^2(\T^d)} &  =  \langle v^1(0), \varphi \rangle_{L^2(\T^d)}  - \frac{1}{2}\int_0^{t} \langle \sumi \gamma_R^2( \| \bar v^1(s)\|_{H^{k}}) \g_\iota'(\bar v^1(s))\Div(\g_\iota(\bar{v}^1(s))), \grad \varphi \rangle_{L^2(\T^d)}  \dd s \\
& -  \int_0^{t}\langle \grad \bar{v}^1(s), \grad \varphi \rangle_{L^2(\T^d)}  \dd s  \\
& - \sumi \int_0^{t}  \gamma_R( \|\bar v^1(s)\|_{H^{k}}) \langle \Div(\g_\iota(\bar{v}^1(s))), \varphi  \rangle_{L^2(\T^d)} \dd \tilde{W}^\iota(s)  \qquad \forall t\geq 0,
 \end{aligned}
 \end{equation*}
which implies (by regularity of $v^1$) the claim.

\bigskip
\textbf{II Convergence of $(u_{\varepsilon_n})_{n\in\N}$ to $u$ in probability:}
We have shown that $(u_{\varepsilon_n})_{n\in\N}$  converges to some limit $v$ in probability. We choose some subsequence $(u_{\varepsilon_{n_l}})_{l\in\N}$  that converges $\PP$ almost surely.
By the same arguments as in the previous step, using the $\PP$-a.s. convergence, we can deduce that $v$ is a variational solution of \eqref{eq:approx-cut-local}. This implies the claim.
\end{proof}

Now we provide the proof of the main theorem:
\begin{proof}[Proof of \cref{thm:main}]
Let $(\varepsilon_n)_{n\in\N}\subseteq (0,1)$ such that $\lim_{n\to\infty} \varepsilon_n = 0$. Choose $\delta, \bar{\delta}\in(0,1)$ arbitrarily small. We show
\begin{equation}\label{eq:proof-of-main-theorem-assertion}
 \limsup_{n\to\infty} \PP\left[ \sup_{t\in[0,\tilde{\sigma}_m\wedge T]} \|u^{\varepsilon_n}(t) - u(t)\|_{H^k(\T^d)} >\delta\right] \leq \bar{\delta}.
\end{equation}
which implies the assertion of the theorem. We shorten the notation by $\tau:= \tilde{\sigma}_m \wedge T$. Define
\begin{equation*}
\sigma_R := \inf\{t\in[0,\tau] \mid \|u(t)\|_{H^k(\T^d)} \geq R\}\wedge\tau,  \qquad
\sigma_{n,R}  := \inf\{t\in[0,\tau]\mid \|u^{\varepsilon_n}(t)\|_{H^k(\T^d)} \geq R+1\}\wedge\tau.
\end{equation*}
Choose $R>0$ such that $\PP\left[ \sigma_R < \tau\right]\leq \bar{\delta}$. 
We have:
\begin{equation*}
\begin{aligned}
\PP\left[ \sup_{t\in[0,\tau]} \|u^{\varepsilon_n}(t) - u(t)\|_{H^k(\T^d)} >\delta\right] & \leq \PP\left[ \sup_{t\in[0,\tau\wedge\sigma_R\wedge\sigma_{n,R}]} \|u^{\varepsilon_n}(t) - u(t)\|_{H^k(\T^d)} >\delta\right] \\
& \quad 
+ \PP\left[ \sigma_{n,R} < \tau \wedge \sigma_R \right] + \PP\left[ \sigma_{R} < \tau\right].
\end{aligned}
\end{equation*}
Uniqueness of local-in-time solutions of \eqref{eq:approx-cut-nonlocal} and \eqref{eq:approx-cut-local}, see \cref{thm:local-in-spaceEquationCutoffWP}  and \cref{thm:nonlocal-in-spaceEquationCutoffWP}, ensures $u^{\varepsilon_n}= u^{\varepsilon_n, R+1}$ and $u = u^{R+1}$ on $[0,\tau\wedge\sigma_R\wedge\sigma_{n,R}]$, where $u^{\varepsilon_n, R+1}$ and $u^{R+1}$ denote the unique solutions of \eqref{eq:approx-cut-nonlocal} and \eqref{eq:approx-cut-local} respectively.
Hence, the convergence of the tamed SPDEs given by \cref{lemma:convergence-of-eq-with-cutoff} implies
\begin{equation*}
\begin{aligned}
& \lim_{n\to\infty}\PP\left[\sup_{t\in[0,\tau\wedge\sigma_R\wedge\sigma_{n,R}]} \|u^{\varepsilon_n}(t) - u(t)\|_{H^k(\T^d)} >\delta\right] \\
& \quad 
\leq \lim_{n\to\infty} \PP\left[\sup_{t\in[0,T]} \|u^{\varepsilon_n,R+1}(t) - u^{R+1}(t)\|_{H^k(\T^d)} >\delta\right] = 0.
\end{aligned}
\end{equation*}
Observing $\|u_0\|_{H^k(\T^d)} < R$ on $\{\sigma_R>0\}$, and $u^{\varepsilon_n}= u^{\varepsilon_n, R+1}$ on $[0,\sigma_{n,R}]$, we may conclude 
\begin{equation*}
 \begin{aligned}
\PP\left[ \sigma_{n,R} < \tau\wedge\sigma_R\right]
& \leq \PP\left[\sup_{t\in[0,\tau\wedge\sigma_R]} \|u^{\varepsilon_n}(t)\|_{H^k(\T^d)} \geq R+1, \|u_0\|_{H^k(\T^d)} < R\right] \\
& = \PP\left[\sup_{t\in[0,\tau\wedge\sigma_R]} \|u^{\varepsilon_n, R+1}(t)\|_{H^k(\T^d)} \geq R+1, \|u_0\|_{H^k(\T^d)} < R\right] \\
& \leq \PP\left[\sup_{t\in[0,T]} \|u^{\varepsilon_n, R+1}(t)-u^{R+1}(t)\|_{H^k(\T^d)} \geq 1/2\right] \\
& \quad +  \PP\left[\sup_{t\in[0,\sigma_R\wedge \tau]} \|u^{R+1}(t)\|_{H^k(\T^d)} \geq R+1/5, \|u_0\|_{H^k(\T^d)} < R\right] \\
 \end{aligned}
\end{equation*}
where the last term is $0$ due to the continuity of $t\mapsto \|u^{R+1}(t)\|_{H^k(\T^d)}$ ensured by \cref{thm:local-in-spaceEquationCutoffWP} and the definition of $\sigma_R$. By
\cref{lemma:convergence-of-eq-with-cutoff} the second last term converges to $0$ for $n\to\infty$. Combining the previous three estimates implies the assertion \eqref{eq:proof-of-main-theorem-assertion}.
\end{proof}


\section{Inviscid setting: Counterexample on entropic limits}
\label{sec:7}
Consider, for $\nu \in [0,\infty)$, $\varepsilon>0$, the family of non-local stochastic partial differential equations
\begin{equation}\label{eq:SPDEnonlocal-lambda}
\begin{aligned}
\dd \ue & = -  \sum_{\iota \geq 1} \Div(\ue g_{\iota}(\ue\falt,x)) \circ \dd W^\iota_t  +  \nu \Delta \ue \dd t, \qquad  u^\varepsilon(0) = u_0.
\end{aligned}
\end{equation}
and, for $\nu \in [0,\infty)$, the corresponding family of local stochastic partial differential equations
\begin{equation}\label{eq:SPDElocal-lambda}
\begin{aligned}
\dd u = -  \sum_{\iota \geq 1} \Div(ug_{\iota}(u,x)) \circ \dd W^\iota_t  + \nu \Delta u \dd t, \qquad  u(0) = u_0.
\end{aligned}
\end{equation}

For $\nu>0$, \cref{thm:main} establishes convergence of the solution $\ue$ of \eqref{eq:SPDEnonlocal-lambda} to the solution $u$ of \eqref{eq:SPDElocal-lambda} as $\varepsilon\to0$, on any time interval on which $u$ has the required regularity. Under suitable regularity and decay assumptions on $(g_\iota)_{\iota\geq1}$, \cref{rmk:globalMaxiSol} guarantees this regularity on every finite time interval for every non-negative initial datum $u_0\in L^2(\Omega,\F_0,\PP;H^k(\T^d))$, provided that $\nu>C_g$, where $C_g$ depends on the coefficients.

The purpose of this section is to show that such convergence may fail in the limiting case $\nu = 0$. More precisely, we present an example of an SPDE and initial data for which $\ue$ does not converge to the entropy solution $u$ when $\nu = 0$. This provides a stochastic analogue of the deterministic result in \cite[Counterexample 5.5]{ColomboCrippaSpinolo}, where the same phenomenon was observed when the stochastic term $\circ \mathrm{d}W$ is replaced by the deterministic differential $\mathrm{d}t$. For simplicity, we present the example on the domain $\mathbb{R}$ rather than on the torus.

\textbf{Proof strategy:} We begin by stating the assumptions. In \cref{lemma:7-2}, we prove that $u^\epsilon$ preserves the $L\log L$ entropy by suitably approximating the equation and applying It\^o's formula to the approximates. We then show in \cref{lemma:7-3} that the $L\log L$ entropy of $u$ is strictly negative for positive times. Combining these results, we conclude in \cref{counterexample} that $u^\varepsilon$ does not converge to the entropy solution $u$.

\textbf{Notation of this section:}
Let $\M$ denote the set of finite (positive) measures on $\R$ equipped with the Borel $\sigma$-algebra. Let $M_r(\R) \subset \M$ denote the space of finite measures $\mu$ with mass $r$, and equip it with the Kantorovich-Rubinstein norm $\|\mu\|_{KR}:= \sup_{\varphi \in \textrm{Lip}_1}(\int_\R \varphi \dd \mu)$ (whereby $\textrm{Lip}_1\subset C^0(\R,\R)$ denotes the set of Lipschitz continuous functions whose supremum norm and Lipschitz constant are bounded by $1$). We frequently use the canonical inclusion of non-negative $L^1(\R)$ functions into $\M$. Let $\mathcal{P}(\R)\subset \M$ denote the set of probability measures on $\R$, and $\mathcal{P}_2(\R):= \{\mu \in \mathcal{P}(\R) \mid \int x^2 \mu(dx) < \infty \}$.  Let $\eta$ be given as in \cref{assumption} and set $\mu\ast\eta(x): = \int_\R \eta(x-y)\dd \mu(y)$ for any $\mu \in \M$. Let $(W(t))_{t\geq 0}$ denote in this section a 1-dimensional Brownian motion. We assume in this section that the underlying filtered probability space $(\Omega, \F, \PP, (F_t)_{t\geq 0})$ has the following structure:
Let $(\Omega, \F,\PP)$ be a probability space with $U\sim\mathrm{Unif}[0,1]$ and let $(W,B)$ be 1-dimensional Brownian motions, all mutually independent. Let $(\F_t)_{t\ge0}$ be the usual augmentation of
\begin{equation*}
\mathcal F_t^0:=\sigma(U)\vee\sigma(W_s,B_s:0\le s\le t).
\end{equation*}
This choice ensures that the assumptions of \cite{CoghiGess} on the filtered probability space are directly satisfied. We say  $\mu\in L^1_\omega C_t \mathcal{M}_r$ if $\mu$ is an adapted continuous process
$\mu \colon [0,T] \times \Omega \to \mathcal{M}_r(\R)$ such that $\E[\sup_{t\in[0,T]}\|\mu_t\|_{KR}]<\infty$. We set for all $x\in \R$, $t\in[0,T]$
\begin{equation*}
\begin{aligned}
\mu \ast \eta(x,t) &:= \int_\R \eta(x-y) \mu_t
( \dd y).
\end{aligned}
\end{equation*}

We consider the non-local SPDE
\begin{equation}\label{eq:cntex-nonlocalSPDE}\tag{$SB_{\varepsilon>0,\nu=0}$}
\dd u^\varepsilon = -\partial_x (u^\varepsilon (u^\varepsilon\ast\eta_{\varepsilon})) \circ \dd W_t \text{ in } \R\times[0,T], \quad \ue(\cdot, 0) = u_0 \text{ on } \R,
\end{equation}
which fits into the framework of stochastic nonlinear non-local Fokker-Planck equations \cite[Definition 5.2]{CoghiGess} by defining
\begin{equation*}
\begin{aligned}
& a \colon  \R_+\times \R \times \M \to \R_{\geq 0}, \qquad 
&& a(t, x, \mu) := \frac{1}{2} (\mu \falt)^2
\\
&b \colon  \R_+\times \R \times \M \to \R, \qquad
&& b(t,x,\mu) := \frac{1}{2}  (\mu \falt) \partial_x (\mu\falt) - \frac{1}{2} \partial_x((\mu\falt) \mu)\falt
\\
& \sigma \colon \R_+\times \R \times \M \to \R, \qquad
&& \sigma(t, x,\mu) := \mu\falt. \\
\end{aligned}
\end{equation*}
The existence and uniqueness of solutions is guaranteed by \cite[Theorem 5.3, Theorem 5.4, Theorem 3.3]{CoghiGess} for $u_0 \in \mathcal{M}_r(\R)$. (To apply these results, the coefficients are first truncated so that the required assumptions are satisfied.)
Moreover, the unique solution $\ue$ is of regularity $L^1_\omega C_t \mathcal{M}_r$.
For details regarding the transformation from the Stratonovich form into the It\^{o} form for $u_0 \in \mathcal{P}_2(\R)$, see \cite[Proposition A.3]{CoghiGess}.

We also consider the local counterpart of the SPDE \eqref{eq:cntex-nonlocalSPDE}
\begin{equation}\label{eq:cntex-localSPDE}\tag{$SB_{\varepsilon=0,\nu=0}$}
\dd u = -\partial_x (u^2) \circ \dd W_t \text{ in } \R\times[0,T], \quad u(\cdot, 0) = u_0 \text{ on } \R
\end{equation}
for which existence and uniqueness of pathwise stochastic entropy solutions in the sense of \cite[Definition 3.1]{LionsPerthameSouganidis} is ensured by \cite[Theorem 3.2]{LionsPerthameSouganidis} for $u_0\in L^1(\R)\cap L^\infty(\R)$.

\begin{lemma}\label{lemma:7-2}
Let $T>0$, $\varepsilon>0$ and let $\eta$ be as defined in \cref{assumption}. Let  $(u^\varepsilon(t))_{t\in[0,T]}$  denote the solution of \eqref{eq:cntex-nonlocalSPDE} with initial condition $u_0 := \one_{[0,1]} \in L^1(\R)\cap L^\infty(\R)\cap M_{1}(\R)$.
Then we have for all $t\in[0,T]$
\begin{equation}\label{eq:cnt-example-proof:claim1}
\int_\R \ue(t)\log(\ue(t)) = \int_\R u_0\log(u_0) \qquad \PP\text{-a.s.}
\end{equation}
\end{lemma}

\begin{proof}[Proof of \cref{lemma:7-2}]
We prove the claim w.l.o.g. for $\varepsilon=1$. We approximate $\ue$ via solutions $\uem$ of an SPDE with frozen coefficients and regularized initial values. Let $(u^m_0)_{m\in\N}$ denote a sequence of $C_c^\infty(\R,[0,\infty))$ functions  such that $\lim_{m\to\infty}\|u^m_0 - u_0\|_{L^2(\R)}=0$, $\|u_0^m\|_{L^1(\R)}= 1$ and $\textrm{supp}(u^m_0) \subset [-1,2]$. By \cite[Definition 4.3, Lemma 4.5, Theorem 4.8]{CoghiGess} we may define the approximations $\uem$ of $\ue$  as the unique solutions of the following frozen SPDE
\begin{equation*}
d \uem = \partial_{xx}^2(a_f \uem) \dd t - \partial_x(b_f \uem)\dd t - \partial_x(\sigma_f \uem) \dd W_t, \quad
\uem_0  = u^m_0
\end{equation*}
where
\begin{equation*}
\begin{aligned}
 \sigma_f(t, x,\omega)  := \sigma(t,x,\ue(t, \cdot, \omega)), \quad a_f (t, x, \omega) := a(t,x,\ue(t, \cdot, \omega)), \quad b_f (t,x,\omega) := b(t,x,\ue(t, \cdot, \omega)).
\end{aligned}
\end{equation*}
In the following we shorten the notation by dropping the $\omega$ in the coefficients.

It can be verified (using e.g. \cite[Theorem 3.3.2, Exercise 3.3.5, Theorem 3.3.3, Theorem 3.4.7, Theorem 4.4.2]{Kunita}) that
\begin{equation}\label{eq:explicitSol}
\uem(t,x) = u_0^m(X_{t,0}(x))Y(X_{t,0}(x),t)
\end{equation}
holds true, whereby $X_{s,t}(x)$ denotes the forward stochastic flow of diffeomorphisms  (ensured by \cite[Theorem 4.6.5]{Kunita}) of the SDE
\begin{equation*}
X_{s,t}(x) = x + \int_s^t \sigma_f(r,X_{s,r}(x)) \circ \dd W_r,
\end{equation*}
$X_{t,s} = X_{s,t}^{-1}$ denotes the inverse flow, and we set
\begin{equation*}
Y(x,t):= \exp\left(-\int_0^t (\partial_x \sigma_f)(s,X_{0,s}(x)) \circ \dd W_s \right).
\end{equation*}
In particular, \eqref{eq:explicitSol} and the smoothness of $X_{0,t}(x)$ and $Y(x,t)$ with respect to $x$ ensure sufficient regularity of $\uem$ for the applications of Ito's formula in the following computations.

A straightforward calculation (omitted here) yields for all $t\in[0,T]$
\begin{equation}\label{eq:lemma73-difference-estimate}
\begin{aligned}
\|u^{\varepsilon,m}(t) - u^{\varepsilon,n}(t)\|_{L^2(\R)}^2
& \leq \|u^{m}_0 - u^{n}_0\|_{L^2(\R)}^2  + \int_0^t C_1 \|u^{\varepsilon,m}(s) - u^{\varepsilon,n}(s)\|_{L^2(\R)}^2  \dd s + M^{m,n}_t
\end{aligned}
\end{equation}
for some continuous local martingale $M^{m,n}$ starting in $0$ and constants $C_1$ and $C_2$ depending on $\eta$ and $\|u_0\|_{L^1(\R)}$. Applying a stochastic Gronwall inequality to \eqref{eq:lemma73-difference-estimate} implies for some constant $C>0$ (independent of $m$)
\begin{equation*}
\begin{aligned}
& \E\left[\sup_{t\in[0,T]} \|u^{\varepsilon,m}(t) - u^{\varepsilon,n}(t)\|_{L^2(\R)}\right]
\leq  C \|u^{m}_0 - u^{n}_0\|_{L^2(\R)} \exp(C_1 T/2).
\end{aligned}
\end{equation*}

Hence, $(\uem)_m$ is a Cauchy sequence in $L^1(\Omega, L^\infty([0,T],L^2(\R)))$, and we denote its limit by $\bar{u}$. It can easily be verified that $\bar{u}$ solves
\begin{equation}\label{eq:frozenSPDE2}
\dd \bar{u}  =  \partial_{xx}(a(t,x,\ue)\bar{u}) - \partial_x(b(t,x,\ue)\bar{u}) - \partial_x(\sigma(t,x,\ue)\bar{u})\dd W_t
\end{equation}
in the sense of \cite[Definition 4.3]{CoghiGess}, noting that by boundedness of
$\cup_m \textrm{supp}_x u^{\varepsilon,m}(t,\omega)$ and  $\textrm{supp}_x \bar{u}(t,\omega)$  for  a.e. $\omega \in \Omega$ for all $t\in[0,T]$, and  $\|u^{\varepsilon,m}(t)\|_{L^1(\R)} =\|u^m_0\|_{L^1(\R)}$, $\|\bar{u}(t)\|_{L^1(\R)} \leq \|u_0\|_{L^1(\R)}$ we also have convergence of $u^{\varepsilon,m}$ to $\bar{u}$ in $L^1(\Omega, L^1([0,T],L^1(\R)))$, hence allowing us to view $\bar{u}$ as an adapted $(B(\R_+)\times \F)$ measurable process $\bar{u}: [0,T]\times \Omega \to M_1(\R)$.
Since $\ue$ also solves \eqref{eq:frozenSPDE2}, uniqueness of solutions \cite[Theorem 4.8]{CoghiGess} yields  $\ue=\bar{u}$.

We can estimate:
\begin{equation*}
\begin{aligned}
 \bigg|\int_{\R} \ue(t) \log(\ue(t)) -\int_{\R} u_0 \log(u_0) \bigg|
 & \leq  \bigg|\int_{\R} \ue(t) \log(\ue(t)) -\int_{\R} \uem(t) \log(\uem(t)) \bigg|  \\
 & \quad +  \bigg|\int_{\R} \uem(t)\log(\uem(t)) -\int_{\R} u^m_0 \log(u^m_0) \bigg|  \\
  & \quad +  \bigg|\int_{\R} u^m_0 \log(u^m_0) -\int_{\R} u_0 \log(u_0) \bigg| \ \\
  & = T_1^m + T_2^m + T_3^m.
 \end{aligned}
\end{equation*}
By the compact support (depending on $\omega$ and $t$) ensured by \eqref{eq:explicitSol} and by convergence in $L^2(\R)$, after passing to a subsequence, which we do not relabel,
\begin{equation*}
 \lim_{m\to\infty }\bigg|\int_{\R} \ue(t) \log(\ue(t)) -\int_{\R} \uem(t) \log(\uem(t)) \bigg| =0 \quad \forall t\in[0,T] \,\, \text{for a.e. }\omega.
\end{equation*}
Using the symmetry of $\eta$, It\^o's formula and the BDG inequality, one can verify that
$\lim_{m\to\infty}\E[|T_2^m|] = 0$. This implies that $T_1^m$, $T_2^m$, and
$T_3^m$ converge to $0$ in probability, and hence the claim follows.
\end{proof}

\begin{lemma}\label{lemma:7-3} 
Let $T>0$. Let $(u(t))_{t\in[0,T]}$  denote the solution of \eqref{eq:cntex-localSPDE} with initial condition $u_0 := \one_{[0,1]} \in L^1(\R)\cap L^\infty(\R)\cap M_{1}(\R)$.
 Then we have
 \begin{equation}\label{eq:cnt-example-proof:claim2}
\PP\left[\int_\R u(T)\log(u(T)) <0\right]=1.
\end{equation}
\end{lemma}

\begin{proof}
For any fixed $T>0$ the entropy solution $u(T)$ has a Lipschitz-continuous representative with probability $1$, by combining \cite[Theorem 2.5]{Karlsen2017} and \cite[Theorem VI.2.3]{RevuzYor}, or by \cite[Theorem 1.2]{GassiatGess}. Moreover, the entropy solution $u$ satisfies the maximum principle and preserves mass. This implies that there exists for a.e. $\omega\in\Omega$ some open interval of $\R$ for which $u(\cdot,T)$ takes values in $(0,1)$, and hence $\int u(T)\log(u(T))dx<0$ with probability one.
\end{proof}

Now we have all lemmas to prove the following counterexample that demonstrates that $\ue$ does not converge to the entropy solution $u$ in the inviscid case $\nu = 0$. It is the stochastic counterpart to \cite[Counterexample 5.5]{ColomboCrippaSpinolo}.
\begin{cntexample}\label{counterexample} Let $T>0$ and  let $\eta$ be as defined in \cref{assumption}. Consider the initial condition $u_0 := \one_{[0,1]} \in L^1(\R)\cap L^\infty(\R)\cap M_{1}(\R)$. For every $\varepsilon>0$ let  $(u^\varepsilon(t))_{t\in[0,T]}$  denote the solution of \eqref{eq:cntex-nonlocalSPDE}, and let $(u(t))_{t\in[0,T]}$  denote the solution of \eqref{eq:cntex-localSPDE}.
Then $\ue \in L^1_\omega C_t \mathcal{M}_{1} \cap L^1([0,T];L^1(\R))\cap L^\infty([0,T];L^2(\R))$ a.s. and for all $p>1$ and $(\varepsilon_{l})_{l\in\N} \subseteq (0,1)$ satisfying $\lim_{l\to\infty}\varepsilon_l = 0$  we have
\begin{equation*}
 \PP\left[\lim_{l\to\infty}\|u^{\varepsilon_l}(T) - u(T)\|_{L^p(\R)} =0 \right] =0.
 \end{equation*}
\end{cntexample}

\begin{proof}[Proof of \cref{counterexample}]
Choose some sequence $(\varepsilon_l)_{l\in\N} \subseteq (0,1)$ such that $\lim_{l\to\infty}\varepsilon_l = 0$.
By \cref{lemma:7-2} and \cref{lemma:7-3}, it follows that
\begin{equation*}
 A := \left\{\omega \in \Omega \,\, \bigg| \,\, \int_\R u(T)\log\big(u(T)\big) <0, \int_\R u^{\varepsilon_l}(T)\log\big(u^{\varepsilon_l}(T)\big) =0 \text{ for all } l\in\N, (W_t)_{t\in[0,T]} \text{ continuous} \right\}
\end{equation*}
is an event of probability $1$. Choose some $\omega\in A$ and drop it in the notation.
Assume that for some $p>1$ we have that $u^{\varepsilon_l}(T)\to u(T)$ in $L^{p}(\R)$ for $l\to\infty$. We observe that $u(T)$ has compact support: Choose $v_0\in C_c^\infty(\R)$ such that $v_0 \geq u_0$. Let $v$ be a solution to \eqref{eq:cntex-localSPDE} with initial condition $v_0$. Then by \cite[Lemma 3.9, Theorem 2.4]{Karlsen2017}, we can deduce that $v(T)$ is compactly supported, by noting it is equal to a solution driven by a finite piece-wise linear path and using the ordinary finite propagation. By the comparison principle, $u(T)$ is therefore also compactly supported. Then by \cite[Lemma 5.6]{ColomboCrippaSpinolo} and our choice of $\omega$ we have
\begin{equation}
 0 > \int_\R u(T)\log\big(u(T)\big) \geq \limsup_{l\to\infty} \int_\R u^{\varepsilon_l}(T)\log\big(u^{\varepsilon_l}(T)\big) = 0
\end{equation}
which is a contradiction. This implies that $\|u^{\varepsilon_l}(T)- u(T)\|_{L^p(\R)}$ does not converge to $0$, and hence the claim.
\end{proof}

\section*{AI disclosure statement.}
No AI was used in the development of proofs. ChatGPT was used to provide an additional check of the proofs, suggest opportunities to shorten the manuscript and improve the language. The authors
reviewed and verified all suggestions and take full responsibility for the
final content.

\section*{Acknowledgments} The first author is supported by the European Union (ERC, FluCo, grant agreement No. 101088488). Views and opinions expressed are however those of the author(s) only and do not necessarily reflect those of the European Union or of the European Research Council. Neither the European Union nor the granting authority can be held responsible for them. The second author acknowledges support by the Max Planck Society through the Research Group “Stochastic Analysis in the Sciences (SAiS)” and the DFG CRC/TRR 388 “Rough Analysis, Stochastic Dynamics and Related Fields”, Project A11.

\bibliographystyle{plain} 
\bibliography{ref} 

\begin{thebibliography}{10}

\bibitem{Aggarwal2025}
Aekta Aggarwal and Ganesh Vaidya.
\newblock Convergence of the numerical approximations and well-posedness:
  nonlocal conservation laws with rough flux.
\newblock {\em Math. Comp.}, 94(352):585--610, 2025.

\bibitem{AgrestiVeraar2022-Part1}
Antonio Agresti and Mark Veraar.
\newblock Nonlinear parabolic stochastic evolution equations in critical spaces
  part {I}. {S}tochastic maximal regularity and local existence.
\newblock {\em Nonlinearity}, 35(8):4100--4210, 2022.

\bibitem{AgrestiVeraar2022-Part2}
Antonio Agresti and Mark Veraar.
\newblock Nonlinear parabolic stochastic evolution equations in critical spaces
  part {II}: {B}low-up criteria and instataneous regularization.
\newblock {\em J. Evol. Equ.}, 22(2):Paper No. 56, 96, 2022.

\bibitem{AgrestiVeraar2024}
Antonio Agresti and Mark Veraar.
\newblock The critical variational setting for stochastic evolution equations.
\newblock {\em Probab. Theory Related Fields}, 188(3-4):957--1015, 2024.

\bibitem{Gess2024}
Diego Alonso-Or\'an, Aythami Bethencourt~de Le\'on, and So~Takao.
\newblock The {B}urgers' equation with stochastic transport: shock formation,
  local and global existence of smooth solutions.
\newblock {\em NoDEA Nonlinear Differential Equations Appl.}, 26(6):Paper No.
  57, 33, 2019.

\bibitem{Amorim2015}
Paulo Amorim, Rinaldo~M. Colombo, and Andreia Teixeira.
\newblock On the numerical integration of scalar nonlocal conservation laws.
\newblock {\em ESAIM Math. Model. Numer. Anal.}, 49(1):19--37, 2015.

\bibitem{Billingsley}
Patrick Billingsley.
\newblock {\em Convergence of probability measures}.
\newblock Wiley Series in Probability and Statistics: Probability and
  Statistics. John Wiley \& Sons, Inc., New York, second edition, 1999.
\newblock A Wiley-Interscience Publication.

\bibitem{Bressan2020}
Alberto Bressan and Wen Shen.
\newblock On traffic flow with nonlocal flux: a relaxation representation.
\newblock {\em Arch. Ration. Mech. Anal.}, 237(3):1213--1236, 2020.

\bibitem{Bressan2021}
Alberto Bressan and Wen Shen.
\newblock Entropy admissibility of the limit solution for a nonlocal model of
  traffic flow.
\newblock {\em Commun. Math. Sci.}, 19(5):1447--1450, 2021.

\bibitem{BrezisMironescu2018}
Ha\"im Brezis and Petru Mironescu.
\newblock Gagliardo-{N}irenberg inequalities and non-inequalities: the full
  story.
\newblock {\em Ann. Inst. H. Poincar\'e{} C Anal. Non Lin\'eaire},
  35(5):1355--1376, 2018.

\bibitem{BrezisMironescu2019}
Ha\"im Brezis and Petru Mironescu.
\newblock Where {S}obolev interacts with {G}agliardo-{N}irenberg.
\newblock {\em J. Funct. Anal.}, 277(8):2839--2864, 2019.

\bibitem{bugini2025nonlinear}
Fabio Bugini, Peter~K. Friz, and Wilhelm Stannat.
\newblock Nonlinear rough {F}okker-{P}lanck equations.
\newblock {\em arXiv preprint arXiv:2507.17469}, 2025.
\newblock Accepted for publication in Mathematical Physics, Analysis and
  Geometry.

\bibitem{Cardaliaguet19}
Pierre Cardaliaguet, Fran{\c{c}}ois Delarue, Jean-Michel Lasry, and
  Pierre-Louis Lions.
\newblock {\em The master equation and the convergence problem in mean field
  games}, volume 201 of {\em Annals of Mathematics Studies}.
\newblock Princeton University Press, Princeton, NJ, 2019.

\bibitem{Chiarello2026}
Felisia~Angela Chiarello, Alexander Keimer, and Lukas Pflug.
\newblock Nonlocal conservation laws with {$p$}-norm, the singular limit
  problem and applications to traffic flow.
\newblock {\em J. Math. Anal. Appl.}, 565(1):Paper No. 130960, 43, 2027.

\bibitem{Chouk2017}
Khalil Chouk and Benjamin Gess.
\newblock Path-by-path regularization by noise for scalar conservation laws.
\newblock {\em J. Funct. Anal.}, 277(5):1469--1498, 2019.

\bibitem{Coclite2025}
Giuseppe~Maria Coclite, Maria Colombo, Gianluca Crippa, Nicola De~Nitti,
  Alexander Keimer, Elio Marconi, Lukas Pflug, and Laura~V. Spinolo.
\newblock Oleinik-type estimates for nonlocal conservation laws and
  applications to the nonlocal-to-local limit.
\newblock {\em Journal of Hyperbolic Differential Equations}, 21(3):681--705,
  2024.

\bibitem{Coclite2023}
Giuseppe~Maria Coclite, Jean-Michel Coron, Nicola De~Nitti, Alexander Keimer,
  and Lukas Pflug.
\newblock A general result on the approximation of local conservation laws by
  nonlocal conservation laws: the singular limit problem for exponential
  kernels.
\newblock {\em Ann. Inst. H. Poincar\'e{} C Anal. Non Lin\'eaire},
  40(5):1205--1223, 2023.

\bibitem{CocliteDeNittiHuang2025}
Giuseppe~Maria Coclite, Nicola De~Nitti, and Kuang Huang.
\newblock Singular limit for a class of nonlocal conservation laws via
  compensated compactness.
\newblock arXiv preprint arXiv:2511.15631, 2025.

\bibitem{Coclite2021}
Giuseppe~Maria Coclite, Nicola De~Nitti, Alexander Keimer, and Lukas Pflug.
\newblock Singular limits with vanishing viscosity for nonlocal conservation
  laws.
\newblock {\em Nonlinear Anal.}, 211:Paper No. 112370, 12, 2021.

\bibitem{CoghiGess}
Michele Coghi and Benjamin Gess.
\newblock Stochastic nonlinear {F}okker-{P}lanck equations.
\newblock {\em Nonlinear Anal.}, 187:259--278, 2019.

\bibitem{Colombo2021}
Maria Colombo, Gianluca Crippa, Elio Marconi, and Laura~V. Spinolo.
\newblock Local limit of nonlocal traffic models: convergence results and total
  variation blow-up.
\newblock {\em Ann. Inst. H. Poincar\'e{} C Anal. Non Lin\'eaire},
  38(5):1653--1666, 2021.

\bibitem{ColomboCrippaMarconiSpinolo2023}
Maria Colombo, Gianluca Crippa, Elio Marconi, and Laura~V. Spinolo.
\newblock Nonlocal traffic models with general kernels: singular limit, entropy
  admissibility, and convergence rate.
\newblock {\em Arch. Ration. Mech. Anal.}, 247:Paper No. 18, 2023.

\bibitem{ColomboCrippaSpinolo}
Maria Colombo, Gianluca Crippa, and Laura~V. Spinolo.
\newblock On the singular local limit for conservation laws with nonlocal
  fluxes.
\newblock {\em Arch. Ration. Mech. Anal.}, 233(3):1131--1167, 2019.

\bibitem{DareiotisGess}
Konstantinos Dareiotis and Benjamin Gess.
\newblock Nonlinear diffusion equations with nonlinear gradient noise.
\newblock {\em Electron. J. Probab.}, 25:Paper No. 35, 43, 2020.

\bibitem{Ito-DebusscheHofmanovaVovelle}
Arnaud Debussche, Martina Hofmanov\'a, and Julien Vovelle.
\newblock Degenerate parabolic stochastic partial differential equations:
  quasilinear case.
\newblock {\em Ann. Probab.}, 44(3):1916--1955, 2016.

\bibitem{Deya2019}
Aur\'elien Deya, Massimiliano Gubinelli, Martina Hofmanov\'a, and Samy Tindel.
\newblock A priori estimates for rough {PDE}s with application to rough
  conservation laws.
\newblock {\em J. Funct. Anal.}, 276(12):3577--3645, 2019.

\bibitem{FehrmanGess2019WP}
Benjamin Fehrman and Benjamin Gess.
\newblock {Well-Posedness of Nonlinear Diffusion Equations with Nonlinear,
  Conservative Noise}.
\newblock {\em Arch. Ration. Mech. Anal.}, 233(1):249--322, 2019.

\bibitem{Fehrman2023}
Benjamin Fehrman and Benjamin Gess.
\newblock Non-equilibrium large deviations and parabolic-hyperbolic {PDE} with
  irregular drift.
\newblock {\em Invent. Math.}, 234(2):573--636, 2023.

\bibitem{FehrmanGess}
Benjamin Fehrman and Benjamin Gess.
\newblock Well-posedness of the {D}ean-{K}awasaki and the nonlinear
  {D}awson-{W}atanabe equation with correlated noise.
\newblock {\em Arch. Ration. Mech. Anal.}, 248(2):Paper No. 20, 60, 2024.

\bibitem{FlandoliGatarek}
Franco Flandoli and Dariusz Gatarek.
\newblock Martingale and stationary solutions for stochastic {N}avier-{S}tokes
  equations.
\newblock {\em Probab. Theory Related Fields}, 102(3):367--391, 1995.

\bibitem{Friedrich2024}
Jan Friedrich, Simone G\"ottlich, Alexander Keimer, and Lukas Pflug.
\newblock Conservation laws with nonlocal velocity: the singular limit problem.
\newblock {\em SIAM J. Appl. Math.}, 84(2):497--522, 2024.

\bibitem{Friedrich2025}
Jan Friedrich, Michael Herty, and Claudia Nocita.
\newblock Control of conservation laws in the nonlocal-to-local limit.
\newblock {\em arXiv preprint arXiv:2510.00677}, 2025.

\bibitem{friz2026mean}
Peter~K. Friz, Ioannis Gasteratos, Ulrich Horst, and Stefanos Theodorakopoulos.
\newblock Mean-field games with rough common noise: the linear-quadratic case.
\newblock {\em arXiv preprint arXiv:2602.19210}, 2026.

\bibitem{Friz2016}
Peter~K. Friz and Benjamin Gess.
\newblock Stochastic scalar conservation laws driven by rough paths.
\newblock {\em Ann. Inst. H. Poincar\'e{} C Anal. Non Lin\'eaire},
  33(4):933--963, 2016.

\bibitem{GassiatGess}
Paul Gassiat and Benjamin Gess.
\newblock Regularization by noise for stochastic {H}amilton-{J}acobi equations.
\newblock {\em Probab. Theory Related Fields}, 173(3-4):1063--1098, 2019.

\bibitem{Gassiat2024}
Paul Gassiat, Benjamin Gess, Pierre-Louis Lions, and Panagiotis~E. Souganidis.
\newblock Long-time behavior of stochastic {H}amilton-{J}acobi equations.
\newblock {\em J. Funct. Anal.}, 286(4):Paper No. 110269, 49, 2024.

\bibitem{ConvexGronwall}
Sarah Geiss.
\newblock Sharp convex generalizations of stochastic {G}ronwall inequalities.
\newblock {\em J. Differential Equations}, 392:74--127, 2024.

\bibitem{GessKonarovskyi23}
Benjamin Gess, Rishabh~S. Gvalani, and Vitalii Konarovskyi.
\newblock Conservative {SPDE}s as fluctuating mean field limits of stochastic
  gradient descent.
\newblock {\em Probab. Theory Related Fields}, 192(3-4):1447--1515, 2025.

\bibitem{Gess2018}
Benjamin Gess and Mario Maurelli.
\newblock Well-posedness by noise for scalar conservation laws.
\newblock {\em Comm. Partial Differential Equations}, 43(12):1702--1736, 2018.

\bibitem{Gess2015}
Benjamin Gess and Panagiotis~E. Souganidis.
\newblock Scalar conservation laws with multiple rough fluxes.
\newblock {\em Commun. Math. Sci.}, 13(6):1569--1597, 2015.

\bibitem{Gess2017}
Benjamin Gess and Panagiotis~E. Souganidis.
\newblock Long-time behavior, invariant measures, and regularizing effects for
  stochastic scalar conservation laws.
\newblock {\em Comm. Pure Appl. Math.}, 70(8):1562--1597, 2017.

\bibitem{Gyongy-Krylov}
Istv\'an Gy\"ongy and Nicolai Krylov.
\newblock Existence of strong solutions for {I}t\^o's stochastic equations via
  approximations.
\newblock {\em Probab. Theory Related Fields}, 105(2):143--158, 1996.

\bibitem{Karlsen2017}
H{\aa}kon Hoel, Kenneth~Hvistendahl Karlsen, Nils~Henrik Risebro, and
  Erlend~Briseid Storr{\o}sten.
\newblock Path-dependent convex conservation laws.
\newblock {\em J. Differential Equations}, 265(6):2708--2744, 2018.

\bibitem{Karlsen2018}
H{\aa}kon Hoel, Kenneth~Hvistendahl Karlsen, Nils~Henrik Risebro, and
  Erlend~Briseid Storr{\o}sten.
\newblock Numerical methods for conservation laws with rough flux.
\newblock {\em Stoch. Partial Differ. Equ. Anal. Comput.}, 8(1):186--261, 2020.

\bibitem{KeimerPflug2019}
Alexander Keimer and Lukas Pflug.
\newblock On approximation of local conservation laws by nonlocal conservation
  laws.
\newblock {\em J. Math. Anal. Appl.}, 475(2):1927--1955, 2019.

\bibitem{Kotelenez95}
Peter Kotelenez.
\newblock A class of quasilinear stochastic partial differential equations of
  {M}c{K}ean-{V}lasov type with mass conservation.
\newblock {\em Probab. Theory Related Fields}, 102(2):159--188, 1995.

\bibitem{Kunita}
Hiroshi Kunita.
\newblock {\em Stochastic flows and stochastic differential equations},
  volume~24 of {\em Cambridge Studies in Advanced Mathematics}.
\newblock Cambridge University Press, Cambridge, 1997.
\newblock Reprint of the 1990 original.

\bibitem{KurtzXiong1999}
Thomas~G. Kurtz and Jie Xiong.
\newblock Particle representations for a class of nonlinear {SPDE}s.
\newblock {\em Stochastic Process. Appl.}, 83(1):103--126, 1999.

\bibitem{KurtzXiong99}
Thomas~G. Kurtz and Jie Xiong.
\newblock A stochastic evolution equation arising from the fluctuations of a
  class of interacting particle systems.
\newblock {\em Commun. Math. Sci.}, 2(3):325--358, 2004.

\bibitem{Lenglart}
{\'E}rik Lenglart.
\newblock Relation de domination entre deux processus.
\newblock {\em Ann. Inst. H. Poincar\'e{} Sect. B (N.S.)}, 13(2):171--179,
  1977.

\bibitem{LionsPerthameSouganidis}
Pierre-Louis Lions, Beno\^it Perthame, and Panagiotis~E. Souganidis.
\newblock Scalar conservation laws with rough (stochastic) fluxes.
\newblock {\em Stoch. Partial Differ. Equ. Anal. Comput.}, 1(4):664--686, 2013.

\bibitem{Lions2014}
Pierre-Louis Lions, Beno\^it Perthame, and Panagiotis~E. Souganidis.
\newblock Scalar conservation laws with rough (stochastic) fluxes: the
  spatially dependent case.
\newblock {\em Stoch. Partial Differ. Equ. Anal. Comput.}, 2(4):517--538, 2014.

\bibitem{LiuRoeckner}
Wei Liu and Michael R\"ockner.
\newblock {\em Stochastic partial differential equations: an introduction}.
\newblock Universitext. Springer, Cham, 2015.

\bibitem{marconi2025nonlocal}
Elio Marconi and Laura~V. Spinolo.
\newblock Well-posedness and singular limit for a nonlocal {GARZ} model.
\newblock {\em Nonlinearity}, 39(8):Paper No. 085009, 25, 2026.

\bibitem{nikolaev2025quantitative}
Paul Nikolaev.
\newblock Quantitative relative entropy estimates for interacting particle
  systems with common noise.
\newblock {\em SIAM J. Math. Anal.}, 57(3):3071--3109, 2025.

\bibitem{nikolaev2025fluctuation}
Paul Nikolaev.
\newblock Fluctuation behaviour for interacting particle systems with common
  noise.
\newblock {\em Stoch. Partial Differ. Equ. Anal. Comput.}, 2026.
\newblock \url{https://doi.org/10.1007/s40072-026-00431-1}.

\bibitem{RevuzYor}
Daniel Revuz and Marc Yor.
\newblock {\em Continuous Martingales and {Brownian} Motion}, volume 293 of
  {\em Grundlehren der mathematischen Wissenschaften}.
\newblock Springer, Berlin, Heidelberg, third edition, 1999.

\bibitem{RocknerShangZhang2024}
Michael R{\"o}ckner, Shijie Shang, and Tusheng Zhang.
\newblock Well-posedness of stochastic partial differential equations with
  fully local monotone coefficients.
\newblock {\em Math. Ann.}, 390(3):3419--3469, 2024.

\bibitem{RocknerShangZhang2026Correction}
Michael R{\"o}ckner, Shijie Shang, and Tusheng Zhang.
\newblock {Correction}: Well-posedness of stochastic partial differential
  equations with fully local monotone coefficients.
\newblock {\em Math. Ann.}, 396:Paper No. 46, 2026.

\bibitem{RotskoffVanden2022}
Grant~M. Rotskoff and Eric Vanden-Eijnden.
\newblock Trainability and accuracy of artificial neural networks: an
  interacting particle system approach.
\newblock {\em Comm. Pure Appl. Math.}, 75(9):1889--1935, 2022.

\bibitem{Simon}
Jacques Simon.
\newblock Compact sets in the space ${L}^p(0,{T}; {B})$.
\newblock {\em Annali di Matematica Pura ed Applicata}, 146:65--96, 1986.

\bibitem{zimper2025mean}
Sebastian Zimper, Ana Djurdjevac, Carsten Hartmann, Christof Sch{\"u}tte, and
  Nata{\v{s}}a~Djurdjevac Conrad.
\newblock Mean-field optimal control with stochastic leaders.
\newblock {\em arXiv preprint arXiv:2512.19201}, 2025.

\end{thebibliography}

\end{document}